\documentclass{amsart}

\usepackage{amsmath, amsfonts, latexsym,amsthm, amsxtra, amssymb,graphicx,cite}

\newtheorem{theorem}{Theorem}[section]
\newtheorem{proposition}{Proposition}[section]
\newtheorem{corollary}{Corollary}[section]

\newtheorem{lemma}{Lemma}[section]
\newtheorem{remark}{Remark}[section]

\font \tenmat = msbm10 \font \sevenmat = msbm7 \font \fivemat =
msbm5 \newfam \matfam \textfont \matfam = \tenmat \scriptfont
\matfam = \sevenmat \scriptscriptfont \matfam = \fivemat \def \mat
{\fam \matfam \tenmat}

\def\NN{{\mat N}} \def\ZZ{{\mat Z}}  \def\RR{{\mat R}} \def\CC{{\mat C}}       

\def\cC{{\mathcal C}}

\def\Res{{\rm Res\,}}\def\Log{{\log\,}}

\title[Pantograph Equation and $\theta$-modular formula]{Analytic continuation of solutions of the pantograph equation by means of $\theta$-modular formula}

\author{Changgui ZHANG}
\address{ Laboratoire P. Painlev\'e (UMR -- CNRS 8524), UFR
Math., Universit\'e de Lille 1, Cit\'e scientifique, 59655
Villeneuve d'Ascq cedex, France}
\email{zhang@math.univ-lille1.fr}
\date{Feb. 2, 2012}

\begin{document}

\begin{abstract}
The aim of this paper is to treat the constant coefficients functio\-nal-differential equation $y'(x)=ay(qx)+by(x)$ with the help of the analytic theory of linear $q$-difference equations. When $ab\not=0$, the associated Cauchy problem with $y(0)=1$ admits a unique power series solution, which is the  Hadamard product  of a usual-hypergeometric series by a basic-hypergeometric series. By means of $\theta$-modular relation, it is proved that this entire function can be expressed as linear combination of all the elements of a system of canonical fundamental solutions at infinity. A family of power series related to values of Gamma function at vertical lines is then introduced, and what really surprises us is that these explicit nonlacunary power series possess a natural boundary. When $a\not=0$ and $b=0$, the asymptotic behavior of solutions will be formulated in terms of the Lambert $W$-function.

\bigskip

{\parindent 0mm {\bf Mathematics Subject Classification 2010:} 34K06, 34M40, 33E30}
\end{abstract}

\maketitle


\numberwithin{equation}{section}

\section*{Introduction}\label{section:introduction}

Let $a$ and $b$ denote two non-zero real or complex numbers. The following func\-tional-differen\-tial equation
\begin{equation}\label{equation:ab}
y'(x)=ay(qx)+by(x)
\end{equation}
has been firstly studied in details in \cite{KM,FMOT} on the seventies of the last century. Since these works, many authors have worked on several extensions to systems or even to some nonlinear or variable-coefficients equations; while not claiming to be exhaustive or complete, we are content for instance to mention \cite{BMW, Ce, Is, Li} and the references therein.  Following \cite{FMOT}, the functional equation \eqref{equation:ab} with real coefficients $a$ and $b$ arises as a mathematical model of an industrial problem involving wave motion in the overhead supply line to an electrified railway system, so Eq. \eqref{equation:ab} is often  called pantograph equation.

As is often the case for the theory of functional $q$-difference equations \cite{DRSZ,Ra0,Zh1,Zh2}, the situation is usually very different and may remain opposed each to other when $0<q<1$ is replaced by $q>1$. This important fact has been developed in \cite{KM}, by stating that the equation \eqref{equation:ab} with the boundary condition $y(0)=1$ is well-posed only if $0<q<1$. In the following, we shall treat only the case $0<q<1$;  see \S  \ref{subsection:KM} for some of remarkable results due to Kato and McLeod \cite{KM} for this case.

Contrary to \cite{KM}, the present paper is devoted to an analysis of the functional-differential equation \eqref{equation:ab} in the complex plane and the corresponding boundary problem will be transformed into one Cauchy problem. Accordingly, the real asymptotic behavior will be replaced with the asymptotic expansion over open sectors of the complex plane.

The analytic theory of linear functional $q$-difference equations is intimately linked with the theory of the elliptic functions; see \cite{Bi, DRSZ, RSZ2}. It may  probably be  natural to make use of elliptic functions to study the functional-differential equation \eqref{equation:ab}, and that is exactly what we shall do in the present paper. More precisely, all analytic solutions of \eqref{equation:ab} will be represented, as often as possible, with the help of Jacobi's theta function, and it will be shown that almost all results given in \cite{KM} under the hypothesis $0<q<1$ can be obtained entirely from our analysis. It is also worth noticing that a family of non-lacunary power series having a natural boundary is deduced from the present work; see \cite{Zh4} and also \S \ref{subsection:naturalboudary} and \S \ref{subsection:KM} in the below.

On the other hand,  as Eq. \eqref{equation:ab} is a functional equation involving both the differential and $q$-difference operators,  it will be shown that its power series solutions are made up of hypergeometric type terms and their $q$-analogs, such as $(\alpha)_n$ and $(\alpha;q)_n$. Therefore, our present work may be seen as a try towards a theory of special functions including both usual- and basic-hypergeometric series. See \eqref{equation:Introalphan} in the below for the notations $(\alpha)_n$ and $(\alpha;q)_n$ and see our paper \cite{Zh1} for a general description of power series satisfying an analytic differential-$q$-difference equation.

The present paper is entirely devoted to the only case of $0<q<1$. We would treat the case of $q>1$ in a future paper. The most important change in this case is that the functional-differential equation \eqref{equation:ab} does not, in general, possess any analytic solution at the neighborhood of the origin nor at the infinity. Indeed, all power series solutions are divergent everywhere and their coefficients have a growth such as $q^{n^2/2}$ or $n!$ as the index $n$ tends to infinity. In order to obtain analytic solutions in sectors of the complex plane, one could apply some very different summation processus, that are Borel-Laplace summation and some of its $q$-analogs; see  \cite{MZ,PRS,Ra,Zh2,Zh3,Zh5}.

 \subsection*{Organization of the paper}\label{subsection:organization}   Until \S \ref{section:b0}, we will always suppose that $b\not=0$ in the functional-differential equation \eqref{equation:ab}; by considering $y(-bx)$ instead of $y(x)$ in  \eqref{equation:ab},  one can suppose that $b=-1$, and this is what we shall do. Therefore, we are led to consider the following equation:
\begin{equation}\label{equation:alpha}
y'(x)=\alpha y(qx)-y(x),\quad
\alpha=q^\mu\in\CC^*,
\end{equation}
where $\mu\in\CC$.

We will start by establishing the fact that every ${\mathcal C}^1$-solution of \eqref{equation:alpha} given on an interval $[0,r)$, $r\in\CC^*$, can be continued into an entire function in the whole complex plan; see Proposition \ref{proposition:Cinfty} in \S \ref{subsection:solution}. Therefore, we limit ourself to the study of the only analytic solution of \eqref{equation:alpha} satisfying $y(0)=1$, and this leads us to the  power series $F(\mu;q,x)$ defined as being a combination of usual and basic-hypergeometric types series; see \eqref{equation:F}. Concerning the point at infinity, one can find a system of solutions of the form $x^{-\mu_k} G_k(\frac1x)$, where $k\in\ZZ$, $\mu_k=\mu-\frac{2\pi k i}{\ln q}$, and where $G_k(z)$ denotes some analytic function at $z=0$ given also as a mix of two types series; see \eqref{equation:G}. 

Understanding the analytic structure of all solutions of \eqref{equation:alpha} constitutes one main objectif of the present paper. For doing that, we shall express $F(\alpha;q,x)$ in terms of all members of the infinite system $\displaystyle \{x^{-\mu_k}G_k(\frac1x)\}_{k\in\ZZ}$, as given in Theorem \ref{theorem:FG}. This allows us to obtain  at infinity an asymptotic expansion of $F$ possessing as coefficients a family of $q$-periodic functions. Surprisingly, these functions can be represented by explicit nonlacunary power series and have a natural boundary. See Theorem \ref{theorem:FPsi} and Remark \ref{remark:why} in the below.

The principal steps we shall follow for proving Theorems \ref{theorem:FG} and \ref{theorem:FPsi}   are inspired by the following observations.
\begin{enumerate}
 \item Written as Dirichlet series, the power series $F(\mu;q,x)$ can be viewed as a Laplace integral in the sense of the theory of $q$-integrals of Jackson; see Proposition \ref{proposition:Fexp} in \S \ref{subsection:Dirichlet}, and Proposition \ref{proposition:Fqintegral} in \S \ref{subsection:qLaplace}. 
 \item\label{LL} Applying Laplace transform to functional equation \eqref{equation:alpha} yields a first order homogeneous $q$-difference equation which admits a Fuchsian singularity at the origin and an irregular singularity at infinity; see \S \ref{subsection:Laplace}.
 \item By considering the classical $\theta$-modular relation as a connection formula for the functional equation $xy(qx)=y(x)$, that relates the analytic solution $\theta(x)$ on $\CC^*$ and the ramified solution $e^{-\log^2x/(2\ln q)}$, the $q$-difference equation obtained in \eqref{LL} admits two types of solutions and they are related via $q$-periodic functions. See Theorem \ref{theorem:Fouriercharacter}  in \S \ref{subsection:characterfunction}.
\end{enumerate}

Accordingly, we are led to consider, in Section \ref{section:2Laplace}, two families of Laplace integrals, one of which represents the solution $F$ and the other, the functions $G_k$. Thanks to the well-known $\theta$-modular formula, we can finish the proof of Theorems \ref{theorem:FG} and \ref{theorem:FPsi} in Section \ref{section:proof}.

The case $b=0$ will be treated in the last two sections. By a simple change of the variable, one can suppose that $a=1$. Instead of the above function $F(\mu;q,x)$, one will make use of its limit form as $b\to 0$ in \eqref{equation:ab}; at the same time, the  system $\displaystyle \{x^{-\mu_k}G_k(\frac1x)\}$ of solutions at infinity does not have limit. We shall introduce a Laplace type integral that defines, in some sense, a canonical solution $h(x)$ at infinity; see Theorem \ref{theorem:b0fh} in \S \ref{subsection:b0connection}. This last solution will be associated to an asymptotic expansion involving the Lambert W-function; see Theorem \ref{theorem:ashW} in \S \ref{subsection:asymp h}.

 \subsection*{Notations}\label{subsection:notations} In what follows, we will denote by $\Log $ the complex logarithm function defined over its Riemann surface  $\tilde\CC^*$, and $x^\alpha=e^{\alpha\Log x}$ for all $\alpha\in\CC$ and $x\in\tilde\CC^*$. As usual, the set $\tilde\CC^*$ will be identified with the product set $]0,\infty[\times\RR$ via the relation $
x=\vert x\vert e^{i\arg x}
$.
Moreover, the following notations will be used.

\begin{itemize}
          \item For all $a$, $b\in\RR$ such that $a<b$,we denote by $S(a,b)$ the open sector of $\tilde\CC^*$ given by
        \begin{equation}\label{equation:IntroSab}
          S(a,b)=\{x\in\tilde \CC^*:a<\arg x<b\}\,.
                 \end{equation}
          By convention, one will make use of the following identification:
          $
          \CC^+= S(-\frac\pi2,\frac\pi2)
          $.
          \item We denote by $\theta(q,x)$ the following Jacobi theta function:
\begin{equation}\label{equation:Introtheta}
\theta(q,x)=\sum_{n\in\ZZ}q^{n(n-1)/2}x^n\,.
\end{equation}
We will write $\theta(x)$ instead of $\theta(q,x)$ if any confusion does not occur.
          \item We denote by $\kappa_q$ or simply $\kappa$ the positive number given by the following relation:
     \begin{equation}\label{equation:kappa}
\kappa=\kappa_q=-\frac{2\pi}{\ln q}>0\,.
\end{equation}
\item For any $\alpha\in\CC$, let $(\alpha)_n$ and $(\alpha;q)_n$ be the sequences given as follows: $
(\alpha;q)_0=(\alpha)_0=1
$, and for $n\ge 1$:
\begin{equation}\label{equation:Introalphan}
 (\alpha;q)_n=\prod_{j=0}^{n-1}(1-\alpha q^j),\quad
(\alpha)_n=\prod_{j=0}^{n-1}(\alpha+j).
\end{equation}
It is clear to see that $(\alpha;q)_n$ can be extended to $(\alpha;q)_\infty$ by taking $n\to\infty$.
        \end{itemize}


\section{Preliminary remarks and statement of results}\label{section:preliminary}

Consider the functional-differential equation \eqref{equation:alpha}, recalled as follows:
$$
y'(x)=\alpha y(qx)-y(x),\quad
\alpha\in\CC^*.\leqno\eqref{equation:alpha}
$$
Until Section \ref{section:b0}, we will fix a $\mu\in\CC$ such that $\alpha=q^\mu$.
Moreover, if $k\in\ZZ$, we will write
\begin{equation}\label{equation:nuk}
\mu_k=\mu+k\kappa i\,,
\end{equation}
where $\kappa$ is as given in \eqref{equation:kappa} in the above.

\subsection{$\cC^\infty$ or analytic solutions}\label{subsection:solution} Let $\Omega$ be a non-empty open  set of $\CC$ or $\RR$ such that $q\Omega\subset \Omega$. For any $x_0\in\Omega$, it follows that $q^nx_0\in\Omega$ for all positive integer $n$, so that $0\in\bar\Omega$. Moreover, if $y$ denotes any given $\cC^1$-solution of \eqref{equation:alpha} on $\Omega$, one may notice that $y$ belongs necessarily to the set $\cC^\infty(\Omega,\CC)$, by taking into account the following relations deduced from \eqref{equation:alpha} by iteration:
\begin{equation}\label{equation:alphan}
y^{(n+1)}(x)=\alpha\,q^n\,y^{(n)}(qx)-y^{(n)}(x),\quad
\forall n\in\NN.
\end{equation}

\begin{lemma}\label{lemma:prelimary}
Let $\Omega$ be a connected open set of $\CC$ such that $q\Omega\subset \Omega\not=\emptyset$, and let $y\in\cC^\infty(\Omega;\CC)$. If $y$ is a solution of \eqref{equation:alpha} such that $\sup_{x\in\Omega,|x|<R}|y(x)|<\infty$ for some $R>0$, then $y$ can be analytically continued into an entire function.
\end{lemma}

\begin{proof}
Firstly, one can notice that the following relation holds for any positive integer $n$ :
\begin{equation}\label{equation:ynqbin}
y^{(n)}(x)=\sum_{k=0}^n(-1)^{n-k}\alpha^{k}\,q^{k(k-1)/2}\,\left[\begin{array}{c}
                                                                        n\cr
                                                                        k
                                                                       \end{array}
\right]_q\,y(q^{k}x)\,,
\end{equation}
where
$$
\left[\begin{array}{c}
                                                                        n\cr
                                                                        k
                                                                       \end{array}
\right]_q=\frac{(q;q)_n}{(q;q)_k\,(q;q)_{n-k}}\,.
$$
This can be easily checked by making use of \eqref{equation:alphan}, and we left the details to the interested reader.

Let $x_0\in\Omega$ be such that $|x_0|<R$. By hypothesis, it follows that $|y(q^nx_0)|\le K<\infty$ for all integer $n\ge 0$, so that,  from \eqref{equation:ynqbin} and \cite[p. 484, (10.0.9)]{AAR}, one obtains easily that
$$\vert y^{(n)}(x_0)\vert \le K\prod_{k=0}^n(1+\vert\alpha\vert q^k)\le K(-\vert\alpha\vert;q)_\infty\,.
$$
Accordingly, we find that $y$ has a Taylor series expansion whose radius of convergence equals to infinity; in other words, $y$ can be analytically continued over the whole complex-plane $\CC$.
\end{proof}

By  almost the same way as what done for Lemma \ref{lemma:prelimary}, one can find the following result.

\begin{proposition}\label{proposition:Cinfty}
Let $d\in\RR$  $\bmod\,2\pi\ZZ$ and let $R>0$. Every given $\cC^\infty$ solution of \eqref{equation:alpha} on $[0,Re^{id})$ can be analytically continued to be an entire function.

Accordingly, Eq. \eqref{equation:alpha}  has no nontrivial $\cC^\infty$-solution on $[0,Re^{id})$ such that $y(0)=0$.
\end{proposition}

\begin{proof}
For any $x_0\in(0,Re^{id})$, by considering the relation \eqref{equation:ynqbin} with $x=x_0$ and by noticing $y(q^nx_0)\to y(0)$ for $n\to\infty$, one finds that $y$ admits a Taylor expansion that converges on the whole plane.
\end{proof}
Consequently, we shall only consider the analytic solutions of \eqref{equation:alpha} on $\CC$.

\subsection{Connection formula between power series-type solutions}\label{subsection:FG} Let $y$ be an analytic function solution to \eqref{equation:alpha} in a neighborhood of $x=0$ in $\CC$. If $a_n=y^{(n)}(0)$ for all integer $n\ge 0$, then putting $x=0$ into Eq. \eqref{equation:alphan} gives raise to the following relation:
$$
a_{n+1}=-(1-\alpha q^n)\,a_n\,.
$$
Therefore, we are led to  the following power series ($q^\mu=\alpha$):
\begin{equation}\label{equation:F}
 F(\mu;q,x)=\sum_{n\ge 0}\frac{(q^\mu;q)_n}{n!}(-x)^n\,.
\end{equation}
Since $(q^\mu;q)_n$ admits a finite limit as $n\to\infty$, one finds that $F(\mu;q,x)$ defines an entire function.

\begin{proposition}\label{proposition:FCinfinity}
Let $\alpha=q^\mu$ as before,  $d\in\RR$ $\bmod$ $2\pi\ZZ$, and $R>0$. Then the entire function $F(\mu;q,x)$   represents the unique  $\cC^\infty$ solution of \eqref{equation:alpha} on $[0,Re^{id})$ such that $y(0)=1$.
\end{proposition}

\begin{proof}
It follows immediately from Proposition \ref{proposition:Cinfty}.
\end{proof}

In order to study the asymptotic behavior of $F(\mu;q,x)$ at infinity, we introduce the following power series:
\begin{equation}\label{equation:G}
G(a;q,x)=\sum_{n\ge 0}\frac{(a)_nq^{n(n+1)/2}}{(q;q)_n}(-x)^n\,,
\end{equation}
which obviously defines an entire function.

One main result that we shall establish in the paper is the following

\begin{theorem}\label{theorem:FG}Let  $\mu\in\CC$, $\alpha=q^\mu\in\CC^*$, and let $\mu_k$ as in \eqref{equation:nuk}. Let $F$ and $G$ be as in \eqref{equation:F} and \eqref{equation:G}, respectively.
Then the following properties hold.
\begin{enumerate}\item The functions $F(\mu;q,x)$ and $\displaystyle x^{-\mu_k}G(\mu_k;q,\frac{1}{x})$, $k\in\ZZ$, all satisfy the func\-tional-differential equation \eqref{equation:alpha}
\item\label{mumu} Moreover, if $\mu\in\CC\setminus(\ZZ_{\le 0}\oplus \kappa\ZZ i)$, then
 for all $x\in\CC^+$, it follows that
\begin{equation}\label{equation:FG}
F(\mu;q,x)=\frac{\kappa\,(q^\mu;q)_\infty}{2\pi\, (q;q)_\infty}\,\sum_{k\in\ZZ}\Gamma(\mu_k)x^{-\mu_k}G(\mu_k;q,\frac{1}{x})\,,
\end{equation}
where $\Gamma$ denotes the usual Euler Gamma function.
\end{enumerate}
\end{theorem}

If $\mu=-m-\epsilon$ and $m\in\NN$, then it follows that
$$
\lim_{\epsilon\to 0}(q^\mu;q)_\infty\,\Gamma(\mu)=(-1)^{m+1}\,\ln q\,(q^{-m};q)_m\,(q;q)_\infty/m!\,,
$$
so that we can observe the following
\begin{remark}\label{remark:FG}\rm
When $\mu\in\ZZ_{\le 0}\oplus \kappa\ZZ i$, the relation \eqref{equation:FG} is reduced to the following one:
$$
F(-m;q,x)=\frac{(q;q)_m}{m!}\,q^{-m(m+1)/2}\,x^m\,G(-m;q,\frac1x)\qquad( m\in\NN),
$$
that can be directly verified as the functions $F$ and $G$ become polynimial.
\end{remark}

\subsection{Natural boundary for the analytic continuation in terms of modular variable}\label{subsection:naturalboudary}
As in \cite{Zh4}, let  $\Psi(u,v,z)$ be the Laurent series of $x$ associated with $(u,v)\in\CC\times\RR$ while $u\notin \ZZ_{\le 0}\oplus\frac{2vi}{\pi}\ZZ$ :
\begin{equation}\label{equation:Psi}
\Psi(u,v,z)=\sum_{n\in\ZZ}\Gamma(u+\frac{2ivn}{\pi})\,z^n\,.
\end{equation}
By the Stirling's formula on $\Gamma$, it follows that $\Psi(u,v,z)$ is convergent over the annulus ${\mathcal C}_\nu$, where
$$
{\mathcal C}_\nu=\{z\in\CC: e^{-\vert v\vert}<\vert z\vert<e^{\vert v\vert}\}\,.
$$
By \cite[Th\'eor\`eme 1]{Zh4}, the function $z\mapsto\Psi(u,v,z)$ can not be analytically continued beyond the borders $\partial{\mathcal C}_\nu$.

Therefore, the relation \eqref{equation:FG} can be stated as follows.

\begin{theorem}\label{theorem:FPsi}
 Let $\mu$ be as in Theorem \ref{theorem:FG} \eqref{mumu}. Then, for all $x\in\CC^+$, it follows that
\begin{equation}\label{equation:FPsi}
F(\mu;q, x)=\frac{\kappa\,(q^\mu;q)_\infty}{2\pi\, (q;q)_\infty}\,\bigl(\frac1x\bigr)^\mu\,\sum_{n\ge 0}\frac{q^{n(n+1)/2}}{(q;q)_n}\,\Psi(\mu+n,\frac{\kappa\pi}{2},x^*)\,\bigl(-\frac{1}{x}\bigr)^n\,,
\end{equation}
where we denote by $x^*$ the modular variable defined as follows:
$$
x^*=x^{-\kappa i}=e^{2\pi i\frac{\log x}{\ln q}}\,.
$$
\end{theorem}

 From the formulas \eqref{equation:FG} and \eqref{equation:FPsi}, one finds that $F(\mu;q,x)=O(x^{-\mu})$ as $x\to\infty$  in the right half plane $\CC^+$. On the other hand, it will be seen that $F(\mu;q,x)$ is exponentially large if $\Re(x)\to-\infty$; see Theorems \ref{theorem:Fexp} and \ref{theorem:FexpN} in Section \ref{section:powerseries}. 
 
 \begin{remark}\label{remark:why}\rm
The formulas \eqref{equation:FG} and \eqref{equation:FPsi} are only valid for $\Re (x)>0$, and this explains why {\bf each function $\Psi(n+\mu,\frac{\kappa\pi}2,x^*)$ has a natural boundary on the imaginary axis $\Re (x)=0$ or, equivalently, on the circles $\vert x^*\vert=e^{\pm\pi^2/\ln q}$}. This  is exactly the subject of \cite[Th\'eor\`eme 1]{Zh4}, which is proved by making use of lacunary Dirichlet series.
  
 \end{remark}

For the proof of Theorems \ref{theorem:FG} and \ref{theorem:FPsi}, see \S \ref{subsection:proofFG}.

 \section{Power series-type solutions and Dirichlet series representation}\label{section:powerseries}

The Dirichlet series expansion techniques are often used for the investigations of pantograph equations; see \cite{FMOT,Is0,Is,MBW}. In the following, the entire function $F(\mu;q,x)$ will be expanded as a Dirichlet series from a point view of $q$-series. This expansion will be used for the study of the asymptotic behaviour of $F(\mu;q,x)$ while $x\to\infty$ inside the left half plane $\Re(x)<0$.

\subsection{A Dirichlet series representation of $F(\mu;q,x)$ }\label{subsection:Dirichlet}
The following expression of $F(\mu;q,x)$ may be known to the researchers of pantograph-type equations, and unfortunately  the author has not found a precise reference about it. However,  the proof we shall give seems somewhat interesting, combining the Hadamard product with a Heine formula for $q$-series.

\begin{proposition}\label{proposition:Fexp}If $q^\mu=\alpha$ and $\Re(\mu)>0$, then the following relation holds for all $x\in\CC$:
\begin{equation}\label{equation:Fexp}
F(\mu;q,x)=(\alpha;q)_\infty\,\sum_{n\ge0}\frac{\alpha^n e^{-q^nx}}{(q;q)_n}\,.
\end{equation}
\end{proposition}

\begin{proof}
By considering $F(\mu;q,x)$ as Hadamard product of the following power series: 
$$\sum_{n\ge 0}(\alpha;q)_nx^n,\qquad\sum_{n\ge 0}\frac{(-1)^n}{n!}x^n,
$$
one can find that
$$
F(\mu;q,x)=\frac{(\alpha;q)_\infty}{2\pi i}\,\sum_{n\ge 0}\int_{\vert t\vert=r<1}\frac1{(q;q)_n}\,\frac{\alpha^n}{1-\frac{q^nx}t}\,e^{-t}\,\frac{dt}t\,,
$$
where the integral is taken over the circle in the anti-clockwise sense. 

Moreover, by hypothesis, $|\alpha|<1$; setting $a=q$, $b=\alpha$ and $c=0$ in \cite[p. 521, Theorem 10.9.1]{AAR} implies that
$$
\sum_{n\ge 0}(\alpha;q)_nx^n=(\alpha;q)_\infty\,\sum_{n\ge 0}\frac1{(q;q)_n}\,\frac{\alpha^n}{1-q^nx}\,.
$$
The wanted relation \eqref{equation:Fexp} is thus obtained by Cauchy's formula.
\end{proof}

The right hand side of \eqref{equation:Fexp} represents a \emph{Dirichlet series} in the sense of \cite[Chapter IX, \S8, p. 432-440]{SZ}.
An alternative proof of Proposition \ref{proposition:Fexp} can be done by checking merely that this Dirichlet series converges really to an analytic solution in the complex $x$-plane  of the Cauchy problem of \eqref{equation:alpha} with the initial condition $y(0)=1$; indeed, such analytic solution is unique.

\begin{remark}\label{remark:Fpolynomial}\rm
The poswer series $F(\mu;q,x)$ becomes a polynomial of $x$ if, and only if, $(q^\mu;q)_\infty=0$, that means exactly that $\mu\in\ZZ_{\le 0}\oplus \kappa i\ZZ$. See also Remark \ref{remark:FG}.
\end{remark}

\subsection{An auxiliary functional equation on $F(\mu;q,x)$}\label{subsection:qalpha}
In order to remove the condition $\Re(\mu)>0$ from Proposition \ref{proposition:Fexp}, one shall make use of the following functional relation:
\begin{equation}\label{equation:qalpha}
\partial_xF(\mu;q,x)=(\alpha-1)\,F(\mu+1;q,x)\,.
\end{equation}
Indeed, one can obtain the last formula from direct computation, be taking into account the following identity: $$(\alpha;q)_{n+1}=(1-\alpha)(q\alpha;q)_n\,,\qquad
\forall n\in\NN.$$

\begin{proposition}\label{proposition:qmalpha} For any positive integer $k$, it follows that
\begin{equation}\label{equation:qmalpha}
\partial_x^kF(\mu;q,x)=(-1)^k(\alpha;q)_k\,F(\mu+k;q,x)\,.
\end{equation}
\end{proposition}
\begin{proof}
Direct calculation by induction on $k$.
\end{proof}

The functional relation \eqref{equation:qalpha} gives raise to a characterization of the function $F(\mu;q,x)$, in view of the following

\begin{proposition}\label{proposition:fqalpha}
Consider an analytic function $f(\alpha,x)$ in $\CC\times\CC$. If there exists an entire function $\alpha\mapsto u(\alpha)$ such that
$$
\partial_xf(\alpha,x)=u(\alpha)f(q\alpha,x),
$$
then $f$ is uniquely determined by its values taken at the complex line $x=0$ in $\CC^2$. More precisely, if we set $f_0(\alpha)=f(\alpha,0)$, then $f$ can be expanded in the following manner:
\begin{equation}\label{equation:fqalpha}
f(\alpha,x)=\sum_{n=0}^\infty \frac{f_n(\alpha)}{n!}\,x^n\,,
\end{equation}
where, for all positive integer $n$,
$$
f_{n}(\alpha)=u(\alpha)\cdots u(q^{n-1}\alpha)\,f_0(q^{n}\alpha)\,.
$$
\end{proposition}

\begin{proof}
One may easily notice that $f_n$'s satisfy the recurrent relation
$$
f_n(\alpha)=u(\alpha)\,f_{n-1}(q\alpha)\,,
$$
which allows us to conclude the proof.
\end{proof}

In the case of $f(\alpha,x)=F(\mu;q,x)$,  relation \eqref{equation:qalpha} implies that $u(\alpha)=\alpha-1$ and $f_0(\alpha)=1$.

\subsection{Asymptotic behaviour of $F(\mu;q,x)$ in the left half-plane} Under the condition $\Re(\mu)>0$, the formula \eqref{equation:Fexp} implies that the first term $(\alpha;q)_\infty\,e^{-x}$ is a dominating term of $F(\mu;q,x)$ as $\Re (x)\to -\infty$. The general case can be treated with the help of Proposition \ref{proposition:qmalpha}, as shown in the following

\begin{theorem}\label{theorem:Fexp}
Let $\alpha=q^\mu\in\CC^*$. The following limit holds  as $x\to\infty$ in the left half-plane $\CC^-$:
  \begin{equation}\label{equation:Fexp-}
 \lim_{\Re( x)\to-\infty} e^x F(\mu;q,x)= (\alpha;q)_\infty\,.
 \end{equation}
Moreover,  for any given open sector $V=S(a,b)$ with $\frac\pi2<a<b<\frac{3\pi}2$, there exists a positive constant $C_V$ such that the following inequality holds for all $x\in V$:
\begin{equation}\label{equation:Fexpestimates}
 \bigl\vert F(\mu;q,x)-(\alpha;q)_\infty\, e^{-x}\bigr\vert<C_V\,e^{-q\Re( x)}\,.
 \end{equation}
\end{theorem}

\begin{proof}Let $V$ be an open sector as given in Theorem \ref{theorem:Fexp}. From Proposition \ref{proposition:Fexp}, we obtain the expected relation \eqref{equation:Fexpestimates} while the condition $\Re(\mu)>0$ is assumed. For an arbitrary complex number $\mu$, choose a enough large positive integer $m$ such that $m+\Re(\mu)>0$, and set $\alpha'=\alpha q^m=q^{\mu'}$ with $\mu'=\mu+m$. Therefore, one can write
$$
F(\mu';q,x)=(\alpha';q)_\infty\,e^{-x}+h(x)\,e^{-qx}\,,
$$
where $h$ denotes a  bounded analytic function over $V$.

Let $\beta=\alpha q^{m-1}=q^{\nu}$, with $\nu=\mu'+1$. From \eqref{equation:qalpha}, it follows that
$$
F(\nu;q,x)=(\beta-1)\,\int_0^xF(\mu';q,t)dt\,,
$$
where the integral is taken over the segment going from the point at origin to the point of affix $x$ in $V$.  An elementary calculation shows that
$$
F(\nu;q,x)=(\beta;q)_\infty\,e^{-x}+H(x)\,e^{-qx}\,,
$$
where
$$
H(x)=-(\beta;q)_\infty\,e^{qx}+(\beta-1)\,\int_0^xh(x-t)\,e^{qt}\,dt\,.
$$
Thus, one finds easily that $F(\nu;q,x)$ satisfies the relation \eqref{equation:Fexpestimates} while replacing $\mu$ by $\nu$; therefore, the function $H(x)$ possesses the same property as $h(x)$ for $F(\mu';q,x)$.
Consequently, one can continue this analysis and obtain finally the relation \eqref{equation:Fexpestimates} for all $\mu\in\CC$.

The relation \eqref{equation:Fexp-} is an evident consequence of \eqref{equation:Fexpestimates}.
\end{proof}

From Remark \ref{remark:Fpolynomial}, if $(q^\mu;q)_\infty=0$, then $F(\mu;q,x)$ becomes a polynomial in $x$, thus one obtains the following
\begin{remark}\label{remark:Fexp} \rm
The function
$F(\mu;q,x)$ is exponentially large for $x\in\CC^-$ if, and only if, $\mu\notin\ZZ_{\le 0}\oplus\kappa i\ZZ$.
\end{remark}

On the other hand, the relation \eqref{equation:Fexpestimates} can be improved to any order $N$ as follows.

\begin{theorem}\label{theorem:FexpN}
 Let $\mu\in\CC$, $\alpha=q^\mu$, and let $V=S(a,b)$ be an open sector such that $\frac\pi2<a<b<\frac{3\pi}2$. Then there exists a positive constant $C=C_V$ such that the following estimates hold for all integer $N\ge 1$ and all $x\in V$:
 \begin{equation}\label{equation:FexpestimatesN}
 \bigl\vert F(\mu;q,x)-(\alpha;q)_\infty\,\sum_{n=0}^{N-1}\frac{\alpha^n}{(q;q)_n} e^{-q^nx}\bigr\vert<C^N\,e^{-q^N\Re (x)}\,.
 \end{equation}
\end{theorem}

\begin{proof}
We omit the proof, which can be done by a similar approach to the proof of Theorem \ref{theorem:Fexp}.
\end{proof}

\subsection{Power series solutions at infinity involving $G(\mu;q,\frac1x)$}\label{subsection:seriesG}

Replacing respectively $\alpha$ and $y$ by $q^\mu$ and $x^{-\mu}\,(1+\sum_{n\ge 1}a_nx^{-n})$ in  \eqref{equation:alpha} leads us to the following relations:
\begin{equation}\label{equation:muan}
a_{n+1}=\frac{\mu+n}{1-q^{-n-1}}\,a_n\,,
\end{equation}
where $n\ge 0$. Thus one finds the following

\begin{proposition}\label{proposition:seriesG}
For any $k\in\ZZ$, let $\mu_k$ be as in \eqref{equation:nuk}. Then $x^{-\mu_k}\,G(\mu_k;q,\frac1x)$ is an analytic solution of \eqref{equation:alpha} in the Riemann surface $\tilde\CC^*$ of the logarithm.
\end{proposition}

\begin{proof}
By replacing $\mu$ with any $\mu_k$ in the second relation of \eqref{equation:muan}, it follows that
$$
a_n=-\frac{(\mu_k+n-1)}{1-q^n}\,(q^n)\,a_{n-1}=...=(-1)^n\,\frac{(\mu_k)_n}{(q;q)_n}\,q^{n(n+1)/2}
$$
for all $n\in\NN$. One gets thus the expression \eqref{equation:G} for the definition of $G(\mu_k;q,\frac1x)$.

It is obvious that $G(\mu_k;q,x)$ converges for all $x\in\CC$, so that $x^{-\mu_k}\,G(\mu_k;q,\frac1x)$ is analytic on the whole surface $\tilde\CC^*$.
\end{proof}

\begin{remark}\label{remark:seriesFG}\rm
The first assertion of Theorem \ref{theorem:FG} follows from Propositions \ref{proposition:FCinfinity} and \ref{proposition:seriesG}.
\end{remark}

\section{Solving \eqref{equation:alpha} by Laplace integrals}\label{section:Laplace}

In \cite{Ma}, Mahler made use of an integral of the type $\displaystyle \int_{\RR}u(t)e^{xq^{it}}dt$ to find special solution for the following functional equation
\begin{equation}\label{equation:Mahler}
y(x+\omega)-y(x)=\omega f(qx),\quad
\omega\not=0.
\end{equation}
Indeed, this integral permits to transform \eqref{equation:Mahler} into a simple difference equation as follows:
$$
u(t+i)=\frac{e^{\omega q^{it}}-1}\omega\,u(t)\,;
$$
which is clearly equivalent to a first order $q$-difference equation if one writes $s=q^{it}$ and $U(s)=u(t)$.

Almost by the same way,  Laplace type integral will be applied to the functional-differential equation \eqref{equation:alpha}, that will be transformed into a first order $q$-difference equation.

\subsection{$F(\mu;q,x)$ viewed as $q$-analogue of Laplace integral}\label{subsection:qLaplace} In the work~\cite{Ja} of F.~H.~Jackson  (see also \cite[\S 10.1]{AAR}, \cite[\S 1]{DZ}), the $q$-integral of a suitable function $f(t)$ over $[0,1]$ is defined as follows:
$$
\int_0^1f(t)\,d_qt=(1-q)\,\sum_{n\ge 0}f(q^n)q^n.
$$
By means of  this discrete integral, we can express $F(\alpha;q,x)$ as a $q$-integral of Laplace type.

\begin{proposition}\label{proposition:Fqintegral}
If $\alpha=q^\mu$ and $\Re(\mu)>0$, then the following relation holds for all $x\in\CC$:
\begin{equation}\label{equation:Fqintegral}
F(\mu;q,x)=\frac{(\alpha;q)_\infty}{(1-q)\,(q;q)_\infty}\,\int_0^1 (qt;q)_\infty\,e^{-tx}\,t^\mu\,\frac{d_qt}t\,.
\end{equation}
\end{proposition}

\begin{proof}Under the assumption, it follows that $\vert \alpha\vert<1$, so that one can express $F(\alpha;q,x)$ by the Dirichlet series \eqref{equation:Fexp}. Thus, putting together $\alpha^n={q^n}^\mu$ and
$$
(q;q)_n=\frac{(q;q)_\infty}{(q\cdot q^{n};q)_\infty}
$$
in the expansion \eqref{equation:Fexp} permits to get the wanted $q$-integral representation \eqref{equation:Fqintegral}.
\end{proof}

\subsection{From \eqref{equation:alpha} to a $q$-difference equation via Laplace transform}\label{subsection:Laplace}
Let $L$ be a smooth loop in the complex $t$-plane and let $qL=\{qt:t\in L\}$ be the loop obtained as the image of $L$ for  the operator $t\mapsto qt$. Consider the following Laplace integral:
\begin{equation}\label{equation:Laplace}
y(x)=\int_Lf(t)e^{-tx}\frac{dt}t\,,
\end{equation}
where $f$ denotes a unknown function.
If we suppose $L$ and $f$ to be chosen such that
\begin{equation}\label{equation:Lq}
\int_{qL}f(t)e^{-tx}\frac{dt}t=\int_Lf(t)e^{-tx}\frac{dt}t,
\end{equation}
then the equation \eqref{equation:alpha} will be transformed as follows:
$$
-tf(t)=\alpha f(\frac tq)-f(t),
$$
or, equivalently,
\begin{equation}\label{equation:f}
(1-qt)f(qt)=\alpha f(t)\,.
\end{equation}

Equation \eqref{equation:f} is Fuchsian at $t=0$ and admits an irregular singular point at $t=\infty$; see \cite{Bi,DRSZ,Ra0,Sa1,Zh2}. If we write $f=gh$, we may decompose \eqref{equation:f} into two $q$-difference equations:
\begin{equation}\label{equation:g}
(1-qt)g(qt)=g(t)
\end{equation}
and \begin{equation}\label{equation:h}
h(qt)=\alpha h(t).
\end{equation}
By iterating \eqref{equation:g}, one obtains easily the following solution:
\begin{equation}\label{equation:seriesg}
g(t)=(qt;q)_\infty=\sum_{n\ge 0}\frac{q^{n(n+1)/2}}{(q;q)_n}(-t)^n,
\end{equation}
which is an entire function with respect to the variable $t$. The power series expansion in \eqref{equation:seriesg} is due to Euler; see \cite[p. 490, Corollary 10.2.2 (b)]{AAR}.

On the other hand, we may make use of several  solutions of \eqref{equation:h} and, by this way, we will get different solutions of \eqref{equation:f}. The choices we will consider are the following:
 \begin{equation}\label{equation:hchoice1}
 h(t)=t^\nu,\quad q^\nu=\alpha,
 \end{equation}
 or
 \begin{equation}\label{equation:hchoice2}
 h(t)=\frac{\theta(\lambda t)}{\theta(\mu t)}\,,\quad
 \frac \mu\lambda=\alpha.
\end{equation}

In Section \ref{section:character}, we shall consider links between the two solutions of \eqref{equation:h} and, in Section \ref{section:2Laplace}, two types of Laplace integrals will be studied.

\section{Remarks on character functions}\label{section:character}

By \cite{Sa1} and \cite{Sa2}, any Fuchsian type linear $q$-difference equation whose coefficients are analytic functions at $x=0$ has a fundamental solution made up of analytic functions in a whole neighborhood of the origin excepted over some $q$-spirals. One main idea consists of making use of the character function $\displaystyle x\mapsto \frac{\theta(\mu x)}{\theta(\lambda x)}$ instead of the multi-valued function $x^\nu$, the last being traditionally used in this domain since Birkhoff \cite{Bi}. Indeed, if $q^\nu=\frac\lambda\mu=a$ and $\sigma_q f(x)=f(qx)$, it follows:
$$
\frac{\sigma_qx^\nu}{x^\nu}=\frac{\sigma_q\frac{\theta(\mu x)}{\theta(\lambda x)}}{\frac{\theta(\mu x)}{\theta(\lambda x)}}=a\,.
$$
In this case, one finds that $\displaystyle x^\nu\,\frac{\theta(\lambda x)}{\theta(\mu x)}$ is $\sigma_q$-invariant or is called to be $q$-periodic.

In the following, we shall make use of $\theta$-modular relation to find the Fourier expansion of such $q$-periodic functions.

\subsection{Character functions expressed by means of $\theta$-modular relation}
For any $x\in\tilde\CC^*$, let
\begin{equation}\label{equation:eqt}
e(q,x)=e(x)=e^{-\frac{\log^2\frac{x}{\sqrt q}}{2\ln q}}\,.
\end{equation}
It is easy to see that both $\theta(x)$ and $e(x)$ satisfy the functional $q$-difference equation $xy(qx)=y(x)$. Moreover, the well-known modular formula on $\theta(q,x)$ says that, if we set
$$
q^*=e^{-2\pi\kappa},\quad
x^*=\iota_q(x)= x^{-\kappa i},
$$
then the following relation holds \cite[p. 498, (10.4.2)]{AAR}:
$$
\theta(q,\sqrt q\,x)=\sqrt\kappa\,e(q,\sqrt q\,x)\,\theta(q^*,\sqrt{q^*}\,x^*)\,
$$
or, equivalently,
\begin{equation}\label{equation:thetamodular}
\theta(q,-x)={\sqrt{\kappa}}\, e(q,-x)\,\theta(q^*,-x^*),
\end{equation}
where $-x=e^{i\pi}x$ in $e(q,-x)$. See \cite{Zh6} for a point of view of $q$-series.

\begin{lemma}\label{lemma:character} The following identity holds for all $\mu\in\CC^*\subset\tilde\CC^*$:
\begin{equation}\label{equation:2theta}
\frac{\theta(q,-{q^\mu} x)}{\theta(q,-x)}=q^{-\mu(\mu-1)/2}\,({e^{\pi i}}x)^{-\mu}\,\frac{\theta(q^*,-{e^{2\pi i\mu}}{x^*})}{\theta(q^*,-{x^*})}\,.
\end{equation}
\end{lemma}

\begin{proof}
It follows directly from \eqref{equation:thetamodular}.
\end{proof}

Remark that if $\mu=n\in\ZZ$, the relation \eqref{equation:2theta} can be read as follows:
\begin{equation}\label{equation:thetafunctionalequation}
\theta(q^nx)=q^{-n(n-1)/2}\,x^{-n}\,\theta(x) \,.
\end{equation}

\subsection{Decomposition of character functions into Laurent series}\label{subsection:characterfunction}
The Jacobi triple  product formula says that
\begin{equation}\label{equation:Jacobitriple}
\theta(x)=(q,-x,-\frac qx;q)_\infty\,.
\end{equation}
Therefore, one finds that for any given $\lambda\in\CC^*\setminus q^\ZZ$, the function $\displaystyle x\mapsto\frac{\theta(-\lambda x)}{\theta(-x)}$ is analytic over $\CC^*\setminus q^\ZZ$.

\begin{lemma}\label{lemma:Laurentseries}
Let $\lambda\in\CC^*\setminus q^\ZZ$ and let $m\in \ZZ$. If $q^m<|x|<q^{m-1}$, then the following Laurent series expansion holds:
\begin{equation}\label{equation:Laurentseries}
\frac{\theta(-\lambda x)}{\theta(-x)}=\frac{\lambda^{1-m}\,\theta(-\lambda)}{(q;q)_\infty^3}\,\sum_{\ell\in\ZZ}\frac{(q^{1-m}x)^\ell}{1-\lambda q^\ell}\,.
\end{equation}
\end{lemma}

\begin{proof}
This can be seen as a special case of Ramanujan's ${}_1\psi_1$-summation formula. Indeed, putting $a=\lambda$, $b=q\lambda$ and replacing $x$ by $q^{1-m}x$ in \cite[p. 502, (10.5.3)]{AAR} yields our wanted formula.
\end{proof}

\subsection{Fourier series expansion of character functions}

We shall conclude this section by proving the following

\begin{theorem}\label{theorem:Fouriercharacter}Let $\mu\in\CC\setminus\ZZ$ and let $m\in\ZZ$.
Then, the following relation holds for all $x\in S(-2m\pi ,2(1- m)\pi)\subset\tilde\CC^*$:
\begin{equation}\label{equation:Fouriercharacter}
\frac{\theta(-q^\mu x)}{\theta(-x)}=C(q,m,\mu)\,x^{-\mu}\,\sum_{\ell\in\ZZ}\frac{e^{2\pi(m-1)\kappa\ell}}{1-e^{2\pi i(\mu+ \kappa i\ell)}}\,x^{-\kappa i\ell},
\end{equation}
where $\kappa$ is given as in \eqref{equation:kappa} and where
\begin{equation}\label{equation:Cqmmu}
C(q,m,\mu)=\frac{\kappa\,(q^\mu,q^{1-\mu};q)_\infty}{i\,(q,q;q)_\infty}\,e^{2(1-m)\pi i\mu}\,.
\end{equation}
\end{theorem}

\begin{proof}
For any $x\in S(-2m\pi,2(1-m)\pi)$, it follows that
$$\Im(\log x)=\arg x\in (-2m\pi,2(1-m)\pi),
$$ so that the following relation holds:
$$\vert x^*\vert=e^{\kappa\,\arg(x)}\in ({q^*}^m,{q^*}^{m-1}),
$$
where $q^*=e^{4\pi^2/\ln q}=e^{-2\pi\kappa}$.
Thus, by Lemmas \ref{lemma:character} and \ref{lemma:Laurentseries}, if one writes
$$
C(q,m,\mu)=\frac{q^{-\mu(\mu-1)/2}\,e^{(1-2m)\pi i\mu}\,\theta(q^*,-e^{2\pi i\mu})}{(q^*;q^*)_\infty^3}\,,
$$
then one gets the following identity:
$$
\frac{\theta(-q^\mu x)}{\theta(-x)}=C(q,m,\mu)\,x^{-\mu}\,\sum_{\ell\in\ZZ}\frac{{q^*}^{(1-m)\ell}}{1-e^{2\pi i\mu}\,{q^*}^\ell}\,{x^*}^\ell\,.
$$
Applying the $\theta$-modular formula \eqref{equation:thetamodular} to $\theta(q^*,-e^{2\pi i\mu})$ yields that
$$
\theta(q^*,-e^{2\pi i\mu})=\frac{q^{1/8}}{i\,\sqrt\kappa}\,e^{\mu\pi i+\kappa\pi/4}\,q^{\mu(\mu-1)/2}\,\theta(q,-q^\mu)\,\,;
$$ thus, by considering the $\eta$-modular relation \cite[p. 538, Theorem 10.12.8]{AAR}:
$$
(q^*;q^*)_\infty=\frac{q^{1/24}}{\sqrt\kappa}\,e^{\kappa\pi/12}\,(q;q)_\infty\,,
$$
 one finds the given expression \eqref{equation:Cqmmu} for $C(q,m,\mu)$. This ends the proof of the expected relation \eqref{equation:Fouriercharacter}.
\end{proof}

\begin{remark}\rm
If $\mu\to n\in\ZZ$, one can notice that
$$C(q,m,\mu)\sim (-1)^n\frac\kappa i\,q^{-n(n-1)/2}\,(1-q^{n-\mu})\,,
$$
so that the relation \eqref{equation:Fouriercharacter} is reduced to the same formula as \eqref{equation:thetafunctionalequation}, by replacing $x$ with  $-x$.
\end{remark}

\section{Two Laplace integrals}\label{section:2Laplace}

Let us come back to the Laplace integral \eqref{equation:Laplace} introduced in \S \ref{subsection:Laplace}. The loop $L$ will be chosen among two types of curves: closed curves near the point at origin, which will be denoted as $\cC$, and half straight-lines starting from the point at origin.

\subsection{Function $I(\alpha;q,x)$}\label{subsection:I}

Let $\alpha\in\CC^*$. Let $\cC$ be any smooth and anti-clockwise Jordan curve whose interior contains the set $ q^\NN=\{1,q,q^2,q^3,...\}$. We consider the function $x\mapsto I(\alpha;q,x)$ defined by the following relation:
$$
I(\alpha;q,x)=\frac{1}{2\pi i}\int_{\cC}\frac{\theta(-\frac\alpha t)}{(\frac1{t};q)_\infty}\,e^{-xt}\frac{dt}{t}\,.
$$
From the analyticity of the function under the integral, one see easily that $I(\alpha;q,x)$ is independent of the choice of the curve $\cC$.

\begin{lemma}\label{lemma:I}
Let $\alpha\in\CC^*$. The  function $x\mapsto I(\alpha;q, x)$ is the unique entire function solution of \eqref{equation:alpha} such that $
y(0)=(\frac q\alpha;q)_\infty$.
\end{lemma}

\begin{proof}
A direct computation shows that $x\mapsto I(\alpha;q,x)$ satisfies the given functional-differential equation \eqref{equation:alpha}. Indeed, let $$ f(t)=\frac{\theta(-\frac\alpha t)}{(\frac1{t};q)_\infty},\quad
h(t)=\frac{\theta(-\frac {qt}\alpha)}{\theta(-qt)}\,.
$$
Thanks to Jacobi's triple product formula \eqref{equation:Jacobitriple}, we find that
$$
f(t)=(q;q)_\infty\,g(t)\,h(t)\,,
$$
where $g(t)$ denotes the function given by \eqref{equation:seriesg} and where $h$ satisfies the $q$-difference equation \eqref{equation:h}. One may also notice that the conditions required for $f$ and $L=\cC$ permit us to transform the functional equation \eqref{equation:alpha} into \eqref{equation:f}.

In order to get the value taken at $x=0$ for $I(\alpha;q,x)$, we may suppose that $\vert\alpha\vert<1$, the general case resulting from a standard analytic continuation argument.

By applying the residues Theorem  to the integral $\displaystyle\int_{\cC}f(t)\frac{dt}{t}$, we find the following relation:
$$
I(\alpha;q,0)=\sum_{n\ge 0}\frac{\theta(-\frac\alpha{ q^n})}{(q^{-n};q)_n(q;q)_\infty}.
$$
From the relation \eqref{equation:thetafunctionalequation} and the fact that
$$
(q^{-m};q)_m=(-1)^m\,(q;q)_m\,q^{-m(m+1)/2}\quad (m\in\NN)
$$
one deduces that
$$
I(\alpha;q,0)=\frac{\theta(-{\alpha})}{(q;q)_\infty}\sum_{n\ge 0}\frac{1}{(q;q)_n}\,(\alpha)^n.
$$
By taking into account the following Euler's identity \cite[p. 490, Corollary 10.2.2~(a)]{AAR}:
\begin{equation}\label{equation:Euler}
\sum_{n\ge 0}\frac{x^n}{(q;q)_n}=\frac{1}{(x;q)_\infty}\quad (\vert x\vert<1)\,,
\end{equation}
one finds finally that
$$
I(\alpha;q,0)=\frac{\theta(-\alpha)}{(q,\alpha;q)_\infty}\,,
$$
which, together with the Jacobi triple  product formula \eqref{equation:Jacobitriple}, allows to complete the proof.
\end{proof}

\begin{proposition}\label{proposition:IF}
The following relation holds for any non-zero complex number $\alpha=q^\mu$:
\begin{equation}\label{equation:IF}
 I(\alpha;q,x)=(\frac{q}\alpha;q)_\infty\,F(\mu;q,x)\,.
\end{equation}
In other words, if $\mu\notin {\NN^*}\oplus\kappa i\ZZ$, then:
\begin{equation}\label{equation:FI}
F(\mu;q,x)=\frac1{(\frac{q}{\alpha};q)_\infty}\,I(\alpha;q,x)\,.
\end{equation}

\end{proposition}

\begin{proof}
By taking into account of Lemma \ref{lemma:I}, one needs only to notice that the function $x\mapsto F(\mu;q,x)$ is the unique function analytic over $\CC$ that satisfies \eqref{equation:alpha} with the condition initial $F(\mu;q,0)=1$. See Proposition \ref{proposition:FCinfinity}.
\end{proof}

\subsection{Function $I_\nu(\alpha;q,x)$}\label{subsection:Inu}

Let $\nu\in\CC$, $\alpha=q^\mu\in\CC^*$ be such that the following inequality holds:
\begin{equation}\label{equation:conditionnu}
\Re(\nu+\mu)>0\,.
\end{equation}
For any real number $d\in(0,2\pi)$, we define
\begin{equation}\label{equation:Inu}
I_\nu^{[d]}(\alpha;q,x)=\int_0^{\infty e^{ id}}\frac{\theta(-\frac \alpha t)}{(\frac{1}{t};q)_\infty}\,e^{-xt}\,t^\nu\,\frac{dt}{t}\,;
\end{equation}
under the condition \eqref{equation:conditionnu}, the integral of \eqref{equation:Inu} converges for all $x$ belonging to the open sector $S(-d-\frac\pi2,-d+\frac\pi2)$ of $\tilde\CC^*$.
Therefore, by the analytic continuation processus, we get an analytic function defined over the sector $S(-\frac{5\pi}2,\frac{\pi}2)$;
this function will be denoted by $I_\nu(\alpha;q,x)$.

By taking into account of the functional equation \eqref{equation:thetafunctionalequation}, one may remark that the following relation holds  for any integer $k\in\ZZ$:
\begin{equation}\label{equation:Inuk}
I_\nu(q^k\alpha;q,x)=\frac{q^{-k(k-1)/2}}{(-\alpha)^k}\,I_{\nu+k}(\alpha ;q,x)\,.
\end{equation}
In particular, when $\nu=0$, the last formula can be read as follows:
\begin{equation}\label{equation:Inu0}
I_k(\alpha;q,x)=(-\alpha)^k\,{q^{k(k-1)/2}}\,I_{0}(q^k\alpha ;q,x)\,.
\end{equation}

\begin{lemma}\label{lemma:Inu}
The function $x\mapsto I_\nu(\alpha;q,x)$ satisfies the following func\-tional-dif\-fe\-ren\-tial equation:
\begin{equation}\label{equation:Inuequation}
 y'(x)+y(x)-q^\nu\alpha y(qx)=0.
\end{equation}
\end{lemma}

\begin{proof}
The result may be proved by a direct computation, as done at the beginning of the proof of Lemma \ref{lemma:I}. See also \S \ref{subsection:Laplace}.
\end{proof}

If we take the derivation with respect to $x$ in the integral \eqref{equation:Inu} of $I_\nu(\alpha;q,x)$, we find that for any positive integer $k$, the following identity holds:
$$
\partial^k_xI_\nu(\alpha;q,x)=(-1)^kI_{\nu+k}(\alpha;q,x).
$$
Thus, from \eqref{equation:Inuk} one deduces the following relation:
\begin{equation}\label{equation:Inukderive}
\partial^k_xI_\nu(\alpha;q,x)=\alpha^k\,q^{k(k-1)/2}\,I_\nu(q^k\alpha;q,x),
\end{equation}
which is similar to that satisfied by $F(\mu;q,x)$; see \eqref{equation:qmalpha}.

\subsection{Two special cases for $I_\nu(\alpha;q,x)$} Let us consider two particular cases: (1) $\nu\in\ZZ$; (2) $\alpha\in q^\ZZ$. The first case contains notably the case of $\nu=0$.

\begin{proposition}\label{proposition:InuF}
Let $\nu=k\in\ZZ$, $\alpha=q^\mu\in\CC^*$ to be such that the condition \eqref{equation:conditionnu} is satisfied. Then, the following relation holds for all $x\in \CC^+=S(-\frac\pi2,\frac\pi2)$:
\begin{equation}\label{equation:InuF}
I_k(\alpha;q,e^{-2\pi i}\,x)-I_k(\alpha;q,x)=C_k(\alpha)\, F(\mu+k;q,x)\,,
\end{equation}
where
$$
C_k(\alpha)=2\pi i\,(-\alpha)^k\,(\frac{q^{1-k}}\alpha;q)_\infty\,q^{k(k-1)/2}\,.
$$
\end{proposition}

\begin{proof}
Notice that when $x\in\CC^+$, both $x$ and $xe^{-2\pi i}$ belong to $S(-\frac{5\pi}2,\frac{\pi}2)$, so the left hand side of \eqref{equation:InuF} is well-defined on $\CC^+$.
By using  the relation \eqref{equation:Inu0}, one can only consider the case of $k=0$. Since
\begin{equation*}\label{equation:Inu0F}
I_0(\alpha;q,x)-I_0(\alpha;q,x e^{-2\pi i})=2\pi i\,I(\alpha;q,x)\,,
\end{equation*}
one completes the proof with the help of Proposition \ref{proposition:IF}.
\end{proof}

\begin{proposition}\label{proposition:alphalambda}
Let $m\in\ZZ$ and $\nu\in\CC$. If $\Re(\nu)+m>0$, then the following relation holds in the sector $ S(-\frac{5\pi}2,\frac{\pi}2)$:
\begin{equation}\label{equation:alphalambda}
I_\nu(q^{m};q,x)=K_\nu(m)\,\bigl(\frac1x\bigr)^{m+\nu}\,G(m+\nu;q,\frac{1}{x})\,,
\end{equation}
where
$$
K_\nu(m)=(-1)^m\,(q;q)_\infty\, q^{-m(m-1)/2}\,{\Gamma(m+\nu)}\,.
$$
\end{proposition}

\begin{proof}
Putting $\alpha=1$ and $k=m$ in \eqref{equation:Inuk} implies that
$$
I_\nu(q^m;q,x)=(-1)^mq^{-m(m-1)/2}I_{\nu+m}(1;q,x),
$$
so that one needs only to prove \eqref{equation:alphalambda} with $m=0$ and $K_\nu(0)=(q;q)_\infty\,\Gamma(\nu)$. Thus we shall suppose that $m=0$ in the statement of Proposition \ref{proposition:alphalambda}.

By making use of the triple product formula \eqref{equation:Jacobitriple} and the Euler's formula \eqref{equation:Euler}, we may write
$$
\frac{\theta(-\frac{1} t)}{(\frac 1 t;q)_\infty}=(q;q)_\infty\,\sum_{n\ge 0}\frac{q^{n(n+1)/2}}{(q;q)_n}\,(- t)^{n}\,.
$$
With the help of Fubini Theorem and the Euler's Gamma function, one may obtain that
$$
\int_0^\infty \frac{\theta(-\frac{1} t)}{(\frac 1 t;q)_\infty}\,e^{-tx}\,t^\nu\,\frac{dt}t= (q;q)_\infty\,\sum_{n\ge 0}(-1)^{n}\frac{\Gamma(\nu+n)}{(q;q)_n}\,q^{n(n+1)/2}\,{x}^{-n-\nu}\,.
$$
The proof is thus completed.
\end{proof}

\begin{remark}\label{remark:Gintegral}\rm
Putting $m=0$ in Proposition \ref{proposition:alphalambda} yields the following integral representation: \begin{equation}\label{equation:Gintegral}
G(\nu;q,\frac1x)=\frac{x^\nu}{\Gamma(\nu)}\int_0^{e^{id}\infty}( qt;q)_\infty\,e^{-xt}\,t^\nu\frac{dt}t\,,
\end{equation}
where $d\in\RR$, $\Re(\nu)>0$ and $x\in S(-\frac\pi2-d,\frac\pi2+d)$.
\end{remark}

In order to get the asymptotic expansion of $I_\nu(\alpha;q,x)$ as $x\to\infty$, we will make use of the relation \eqref{equation:2theta}, which is reduced from the $\theta$-modular formula \eqref{equation:thetamodular}.

\section{End of the Proof of Theorems \ref{theorem:FG} and \ref{theorem:FPsi}}\label{section:proof}

For any $\delta\in\RR$ and $q\in(0,1)$, we denote by $D_{\delta,q}$ or simply $D_\delta$ the following annulus:
\begin{equation}\label{equation:Ddelta}
D_{\delta,q}=D_{\delta}:=\{z\in\tilde\CC^*: e^{-3\kappa\pi/2}<\vert z e^{\delta}\vert<e^{3\kappa\pi/2}\},
\end{equation}
where $\kappa=\kappa_q$.

For any $u\in\CC^+$, $\mu\in\CC\setminus\ZZ$, consider the function $z\mapsto \Phi(u,\mu;q,z)$ given by the following relation:
\begin{equation}\label{equation:Phi}
\Phi(u,\mu;q,z)=z^{-\frac{\mu i}{\kappa}}\,\sum_{\ell\in\ZZ}\frac{\Gamma(u+\mu+\kappa i\ell)}{1-e^{2\pi i(\mu+\kappa i\ell)}}\,z^\ell\,.
\end{equation}
Since
$$
\Gamma(u+\mu\pm \kappa i\ell)=O(\ell^{u+\mu-\frac12}\,e^{-\frac{\kappa \pi}2\ell})
$$
for $\ell\to+\infty$, the Laurent series of \eqref{equation:Phi} converges in the domain $D_{\kappa\pi}$. Consequently,  $\Phi(u,\mu;q,z)$ represents an analytic function on the annulus $D_{\kappa\pi}$ of $\tilde\CC^*$.

Consider the function $\Psi(u,v,x)$ given in \eqref{equation:Psi}, which is related with $\Phi(u,\mu;q,z)$  in the following manner.

\begin{proposition}\label{proposition:PhiPsi} The following relation holds for all $z\in D_{-\kappa\pi}\cap D_{\kappa\pi}$:
\begin{equation}\label{equation:PhiPsi}
\Phi(u,\mu;q,z)-\Phi(u,\mu;q,z\,e^{-2\kappa\pi})=z^{-\frac{\pi i}{\kappa}}\,\Psi(u+\mu,\frac{\kappa\pi}{2},z)\,.
\end{equation}

\end{proposition}

\begin{proof}
It follows form a direct computation, by making use of the definition  \eqref{equation:Psi}  of $\Psi$.
\end{proof}

By considering the fact that $\Psi(u,v,x)$ has a natural boundary (see \cite{Zh4}), one may notice that $\Phi(u,\mu;q,z)$ can not be analytically continued beyond the boundaries of his convergence ring $D_{\kappa\pi}$ in the Riemann surface $\tilde\CC^*$.

\subsection{Expand $I_\nu(\alpha;q,x)$ by means of $\Phi(u,\mu;q,z)$}

Let $\alpha=q^\mu$, with $\mu\in\CC$; we will consider the bahaviour of $I_\nu(q^\mu;q,x)$ as $x\to \infty$. The main result of this section is the following

\begin{theorem}\label{theorem:InuPhi}
Let $\mu\in\CC$ and $\nu\in\CC$ be such that $\Re(\mu+\nu)>0$; let $\Phi(u,\mu;q,z)$ be the function given in \eqref{equation:Phi}. If $q^\mu\notin q^\ZZ$, then the following relation holds for all $x$ belonging to the open sector $S(-\frac{5\pi}2,\frac{\pi}2)$:
\begin{equation}\label{equation:InuPhi}
I_\nu(q^\mu;q,x)=C(q,\mu)\,\bigl(\frac1x\bigr)^{\nu}\sum_{n\ge 0}\frac{q^{n(n+1)/2}}{(q;q)_n}\,\Phi(n+\nu,\mu;q,{x^*})\,\bigl(-\frac1x\bigr)^n\,,
\end{equation}
where $C(q,\mu)$ denotes a constant given by means of $C(q,m,\mu)$ of Theorem \ref{theorem:Fouriercharacter} in the following manner:
$$
C(q,\mu)=\frac{\kappa\,(q^\mu,q^{1-\mu};q)_\infty}{i\,(q;q)_\infty}\,.
$$
\end{theorem}

\begin{proof}
Consider the integral \eqref{equation:Inu} of $I_\nu(\alpha;q,x)$ and write
$$
\frac{\theta(-\frac{q^\mu} t)}{(\frac1t;q)_\infty}=(q;q)_\infty\,g(t)\,h(t),
$$
where, as in the proof of Lemma \ref{lemma:I} ($\alpha=q^\mu$), we set
$$
g(t)=(qt;q)_\infty,\quad
h(t)=\frac{\theta(-\frac{q^\mu} t)}{\theta(-\frac1t)}\,.
$$
We apply Theorem \ref{theorem:Fouriercharacter} to expand $h(t)$ into a Fourier series for $t\in S(0,2\pi)$: since $\frac1t\in S(-2\pi,0)$, putting $m=1$ in \eqref{equation:Fouriercharacter} allows us to obtain the following expression:
$$
h(t)=\frac{\kappa\,(q^\mu,q^{1-\mu};q)_\infty}{i\,(q,q;q)_\infty}\,\sum_{\ell\in\ZZ}\frac{t^{\mu+\kappa i\ell}}{1-e^{2\pi i(\mu+\kappa i\ell)}}\,,
$$
where $q^*$ was replaced by $e^{-2\pi \kappa}$ and $\kappa=\kappa_q=-\frac{2\pi}{\ln q}$.

Therefore, from the Euler's relation \eqref{equation:seriesg} it follows that
\begin{equation}\label{equation:Fouriert}
\frac{\theta(-\frac{q^\mu} t)}{(\frac1t;q)_\infty}=C(q,\mu)\,\sum_{n\ge0}\sum_{\ell\in\ZZ}\frac{(-1)^nq^{n(n+1)/2}}{(q;q)_n}\,\frac{t^{n+\mu+\kappa i\ell}}{1-e^{2\pi i(\mu+\kappa i\ell)}}\,,
\end{equation}
where
$$C(q,\mu)=\frac{\kappa\,(q^\mu,q^{1-\mu};q)_\infty}{i\,(q;q)_\infty}\,.$$
In \eqref{equation:Fouriert}, the double series indexed by $n$ and $\ell$ is normally convergent on any compact of $S(0,2\pi)$. At the same time, in view of the relation $x^*=e^{-\kappa i}$, one may notice that
$$
\int_0^\infty t^{n+\mu+\nu+\kappa i\ell}\,e^{-tx}\,\frac{dt}t=\Gamma({n+\mu+\nu+\kappa i\ell})\,\bigl({x^*}\bigr)^{\ell-\frac{\mu i}\kappa}\,\bigl(\frac1x\bigr)^{n+\nu}\,.
$$
 Hence, if one considers the expansion \eqref{equation:Fouriert} in the integral \eqref{equation:Inu} and makes use of the termwise integration for each $e^{-tx}\,t^\gamma$, one can obtain finally the formula \eqref{equation:InuPhi}, according to the Lebesgue's dominated convergence Theorem. The proof of Theorem \ref{theorem:InuPhi} is thus completed.
\end{proof}

Theorem \ref{theorem:InuPhi} states a remarkable fact in relation with the asymptotic bahaviour at infinity of the function $I_\nu(q^\mu;q,x)$: \emph{it can be expanded as a power series of $\frac1x$ having $q$-periodic functions as coefficients}. This phenomenon will also occur for $F(\alpha;q,x)$ and other functions.

\subsection{End of the Proof of Theorem \ref{theorem:FG} and \ref{theorem:FPsi}}\label{subsection:proofFG} The assertion (1) of Theorem \ref{theorem:FG} can be easily checked; see Remark \ref{remark:seriesFG}.
The relations \eqref{equation:FG} and \eqref{equation:FPsi} can be obtained directly one from other, so
we shall make use of Proposition \ref{proposition:InuF} and of Theorem \ref{theorem:InuPhi} to conclude only the proof of Theorem \ref{theorem:FPsi}.

Let $x\in S(-\frac\pi2,\frac\pi2)$ and let $\mu\in\CC^+$ such that $q^\mu\notin q^\ZZ$. In view of the relation $(xe^{-2\pi i})^*=x^*\,q^*=x^*\,e^{-2\pi\kappa}$, putting together the formulas  \eqref{equation:PhiPsi} and \eqref{equation:InuPhi} implies the following identity:
\begin{equation*}\label{equation:Inu0Phi}
I_0(q^\mu;q,x)-I_0(q^\mu;q,xe^{-2\pi i})=C(q,\mu)\,\bigl(\frac1x\bigr)^\mu\sum_{n\ge 0}\frac{q^{n(n+1)/2}}{(q;q)_n}\,\Psi(n+\mu,\frac{\kappa\pi}{2},{x^*})\,\bigl(-\frac1x\bigr)^n\,,
\end{equation*}
where $C(q,\mu)$ is the constant defined in Theorem \ref{theorem:InuPhi}

By letting $k=0$ and $\alpha=q^\mu$ in the relation \eqref{equation:InuF}, one finds finally that
$$
F(q^\mu;q,x)=C_0(q,\mu)\,\bigl(\frac1x\bigr)^\mu\sum_{n\ge 0}\frac{q^{n(n+1)/2}}{(q;q)_n}\,\Psi(n+\mu,\frac{\kappa\pi}{2},{x^*})\,\bigl(-\frac1x\bigr)^n\,,
$$
where
$$
C_0(q,\mu)=-\frac{C(q,mu)}{C_0(q^\mu)}=\frac{\kappa\,(q^\mu;q)_\infty}{2\pi\, (q;q)_\infty}\,.
$$

Remark that one can remove the restriction $\Re(\mu)>0$ from the above-done analysis (see \eqref{equation:qalpha}), by reasoning with a standard analytic continuation processus. Thus the proof of Theorem \ref{theorem:FPsi}, and therefore that of Theorem \ref{theorem:FG}, are achieved.\hfill $\Box$

\subsection{Revisit one Theorem due to Kato and McLeod}\label{subsection:KM}
By Theorem \ref{equation:FG}, one can given more precision to the following result, that constitutes probably one of the most important steps for the investigations of the asymptotic behaviour of solutions of the functional-differential equation \eqref{equation:ab}.

\begin{theorem}[Theorem 3, \cite{KM}]\label{theorem:KM} Consider the boundary problem associated with equation \eqref{equation:ab} for $0\le x<\infty$, and the boundary condition  $y(0)=1$, and suppose that $0<q<1$, $a\in\CC^*$ and $b<0$. Let $\mu$ to be a complex number such that $q^{\mu}=-a/b$. Then the following assertions hold.
\begin{enumerate}
 \item There exists no solution $y(x)$ such that $y(x)=o(x^{-\Re(\mu)})$ as $x\to\infty$.
 \item Every solution $y(x)$ is $O(x^{-\Re(\mu)})$ at the infinity and may be written as follows:
\begin{equation}\label{equation:introKM}
y(x)=x^{-\mu}\bigl\{\sum_{n=0}^\infty\frac{q^{n(n+1)/2}}{(q;q)_n}\,g_n(\log x)\,(-\frac1{bx})^n\bigr\}\,,
\end{equation}
where $g_0=g$ denotes some ${\mathcal C}^\infty(\RR;\CC)$-periodic function of period $|\log q|$ verifying
\begin{equation}\label{equation:introKMg}
|g^{(n)}(s)|\le K^nq^{-n^2/2}, \quad
\forall n\in\NN
\end{equation}
for some constant $K>0$, and where all the functions $g_n$ are recursively given by the following relation:
\begin{equation}\label{equation:introKMgn}
g_{n+1}=g_n'-(\mu+n)g_n\,.
\end{equation}
\end{enumerate}

\end{theorem}

Indeed, if one writes $f(x)=y(-bx)$, then $f$ will satisfies the boundary problem about the equation \eqref{equation:alpha} with $\alpha=-a/b$, $0\le x<\infty$ and $f(0)=1$. By Proposition \ref{proposition:FCinfinity}, the function $f$ is {\it unique} and is necessarily represented by $F(\mu;q,x)$. Let $s=\log x$ and write $x^*=e^{-i\kappa s}$ in relation \eqref{equation:FPsi} of Theorem \ref{theorem:FG}; one finds that the functions $g_n$, $n\ge 0$, appeared in  \eqref{equation:introKM} can be  defined as follows:
$$
g_n(s)=(-1)^n\,\frac{\kappa\,(q^\mu;q)_\infty}{2\pi\, (q;q)_\infty}\,\Psi(n+\mu,\frac{\kappa\pi}{2},e^{-i\kappa s})\,.
$$
Therefore, one can easily get the conditions \eqref{equation:introKMg} and \eqref{equation:introKMgn} by the definition \eqref{equation:Psi} of $\Psi$; see also \cite[\S 1.2]{Zh4} for the functional equation \eqref{equation:introKMgn}.

\section{Solutions of $y'(x)=y(qx)$ and their integral representation}\label{section:b0}

The reste of this paper is devoted to the degenerate case with $b=0$, that has been considered in several works, in particular in \cite{MFB} and \cite[\S2]{Is0}. Letting $b=0$ in the functional-differential \eqref{equation:ab}, the only non-trivial case is $a\not=0$; by substituting $y(x)$ by $y(ax)$, one can suppose that $a=1$, so that the equation we shall consider is the following:
\begin{equation}\label{equation:a1b0}
y'(x)=y(qx).
\end{equation}

\subsection{Integral representation of solutions}\label{subsection:b0integral} 

An direct computation shows that if one denotes
\begin{equation}\label{equation:b0fseries}
f(x)=f(q,x)=\sum_{n\ge0}\frac{q^{n(n-1)/2}}{n!}\,x^n\,,
\end{equation}
then $f(x)$ is the unique entire function that satisfies both the functional-differential equation \eqref{equation:a1b0} and the initial condition $y(0)=1$. It is obvious to see that the following integral representation holds:
\begin{equation}\label{equation:b0f}
f(x)=\frac1{2\pi i}\,\int_{\vert t\vert=R>0}e^t\,\theta(\frac xt)\,\frac{dt}t\,,
\end{equation}
where the integral is taken on any smooth closed-loop that encircles the origin once in the positive direction.

Let $g$ be function defined on $\tilde\CC^*$ by the following integral:
\begin{equation}\label{equation:b0g}
g(x)=g(q,x)=\frac1{2\pi i}\,\int_{\RR+0i}q^{t(t+1)/2}\,\Gamma(t)\,(xe^{-\pi i})^{-t}\,dt\,,
\end{equation}
where $\RR+0i$ denotes a straight line from $-\infty+\epsilon i$ to $\infty+i\epsilon$ with $\epsilon>0$. One can notice that $g$ is defined and analytic on the whole Riemann surface $\tilde\CC^*$ and that this function is independent of the choice of $\epsilon>0$. By the same way, we define
\begin{equation}\label{equation:b0g-}
g_-(x)=g_-(q,x)=\frac1{2\pi i}\,\int_{\RR-0i}q^{t(t+1)/2}\,\Gamma(t)\,(xe^{-\pi i})^{-t}\,dt
\end{equation}
for all $x\in\tilde\CC^*$.

\begin{proposition}\label{proposition:b0fg-}
Let $f$, $g$ and $g_-$ as above. Then all these functions satisfy the functional-differential equation \eqref{equation:a1b0}. Moreover, the following identity holds for all non-zero complex $x\in\CC^*$:
\begin{equation}\label{equation:b0fg-}
f(x)=g_-(x)-g(x) \,.
\end{equation}
\end{proposition}

\begin{proof}
Straightforward verification. The relation \eqref{equation:b0fg-} yields by applying the Re\-sidues Theo\-rem, with the help of the classical fact that
$$
\Res(\Gamma(t):t=-n)=\frac{(-1)^n}{n!}\,,\quad
\forall n\in\NN.
$$
\end{proof}

The main goal for the following is to find a family of solutions at infinity in which one may represent the solution $f(x)$. For doing that, we will consider the following function:
\begin{equation}\label{equation:b0h}
h(x)=h(q,x)=\frac{\sqrt\kappa}{2\pi i}\,\int_{\cC}e^{t-\frac1{2\ln q}\log^2\frac{\sqrt q\, t}x}\,\frac{dt}t\,,
\end{equation}
where $x\in\tilde\CC^*$ and where $\cC$ denotes a smooth contour that starts at infinity on the negative real axis, encircles the origin once in the positive direction, and returns to negative infinity. Such contour is  traditionally called Hankel contour and is  used to represent $\frac1{\Gamma(x)}$ by means of the integral of $e^t\,t^{-x}$ (see \cite[p. 51, Exercise 22]{AAR}).

\begin{theorem}\label{theorem:b0gh} The function $h$ satisfies the functional-differential equation \eqref{equation:a1b0} and, moreover, the following identity holds for all $x\in\tilde\CC^*$:
\begin{equation}\label{equation:b0gh}
h(x)=g(x)-g(xe^{2\pi i}).
\end{equation}
\end{theorem}

\begin{proof}
By Proposition \ref{proposition:b0fg-}, $g(x)$ and $g(xe^{2\pi i})$ are solution to \eqref{equation:a1b0}, so it suffices to prove \eqref{equation:b0gh}. By the Euler reflection formula, one writes
$$
\Gamma(t)\bigl(e^{\pi i t}-e^{-\pi i t}\bigr)=\frac{2\pi i}{\Gamma(1-t)}\,;
$$
therefore, from \eqref{equation:b0g} one finds that
$$
g(x)-g(xe^{2\pi i})=\int_{\RR+0i}q^{t(t+1)/2}\,\frac{x^{-t}}{\Gamma(1-t)}\,dt\,.
$$
Thus, by making use of the Hankel contour $\cC$ to represent $1/\Gamma(1-t)$,  one may deduce that
\begin{equation}\label{equation:b0gg}
g(x)-g(xe^{2\pi i})=\frac1{2\pi i}\int_{\RR+0i}\int_\cC q^{t^2/2}\,\bigl(\frac {\sqrt q\,s}x\bigr)^{t}\,e^s\,\frac{ds}s\,dt\,.
\end{equation}

Apply Fubini Theorem to \eqref{equation:b0gg}, take the integration over $\RR$ for the variable $t$ and remember that for all $\alpha>0$ and all $\beta\in\CC$:
$$
\int_\RR e^{-\alpha t^2/2+\beta t}\,dt=\frac{e^{\beta^2/(2\alpha)}}{\sqrt\alpha}\int_\RR e^{-  t^2/2}\,dt=\sqrt{\frac{2\pi}\alpha}\,{e^{\beta^2/(2\alpha)}}\,.
$$
Write $\alpha=-\ln q$, $\kappa=\frac{2\pi}\alpha$, and $\beta=\log\frac{\sqrt q\,s}x$; from \eqref{equation:b0gg} one obtains the integral representation \eqref{equation:b0h} and therefore the expected relation \eqref{equation:b0gh}.
\end{proof}

\subsection{Connection formula between the origin and infinity}\label{subsection:b0connection} The entire function $f$ is obviously the canonical solution of the functional-differential \eqref{equation:a1b0} at the origin. On the other hand, the function $h$ given in \eqref{equation:b0h} may be seen as {\it canonical} solution of \eqref{equation:a1b0} at infinity, in view of the {\it simplicity} of the form of its asymptotic expansion; see Theorem \ref{theorem:ashW} in the below.  By making use of the $\theta$-modular relation \eqref{equation:thetamodular}, one can establish the following
\begin{theorem}\label{theorem:b0fh}
The following relation holds for all non-zero complex number $x$:
\begin{equation}\label{equation:b0fh}
f(x)=\sum_{n\in\ZZ}h(xe^{2n\pi i})\,,
\end{equation}
where the series in the right-hand side is uniformly convergent on any compact of $\CC^*$ as $n\to\pm\infty$.
\end{theorem}

\begin{proof}
Let $\cC$ be an integration contour such as given in \eqref{equation:b0h}. By replacing the circle $|t|=R>0$ with $\cC$ in \eqref{equation:b0f} and by taking into account the functional relation $\theta(x)=\theta(q/x)$, one obtains that
\begin{equation}\label{equation:b0fC}
f(x)=\frac1{2\pi i}\,\int_{\cC}e^t\,\theta(\frac {qt}x)\,\frac{dt}t\,.
\end{equation}
Write the modular relation \eqref{equation:thetamodular} as follows:
$$
\theta(\sqrt q\,z)=\sqrt\kappa\,e^{-\frac{\log^2z}{2\ln q}}\,\sum_{n\in\ZZ}{q^*}^{n^2/2}\,e^{-\kappa i\log z}\,,
$$
where $q^*=e^{4\pi^2/\ln q}$ and $\kappa=-\frac{2\pi}{\ln q}$; it follows that
$$
\theta(\sqrt q\,z)=\sqrt\kappa\,\sum_{n\in\ZZ}e^{-\frac{\log^2(ze^{2\pi i n})}{2\ln q}}\,.
$$
By putting this last relation into \eqref{equation:b0fC} and by checking directly the convergence of the integral and summation, we end the proof.
\end{proof}

\begin{corollary}\label{corrolary:b0fg}
The following limit holds for all $x\in\CC^*$:
\begin{equation}\label{equation:b0fgmn}
f(x)=\lim_{n,m\to+\infty}\bigl(g(xe^{-2\pi ni})-g(xe^{2\pi mi})\bigr)\,,
\end{equation}
where $n$, $m\in\ZZ$.

Consequently, it follows that
\begin{equation}\label{equation:b0fgn}
f(x)=\lim_{\ZZ\in n\to+\infty}2i\int_{\RR+0i}q^{t(t+1)/2}\,\Gamma(t)\,\sin[(2n-1)\pi t]\,x^{-t}\,dt\,.
\end{equation}
\end{corollary}

\begin{proof}
Relation \eqref{equation:b0fgmn} follows immediately from \eqref{equation:b0gh} and \eqref{equation:b0fh}.

Thus, if one considers the integral representation \eqref{equation:b0g} of $g$ and lets $m=n-1$ in \eqref{equation:b0fgmn}, then one obtains the limit \eqref{equation:b0fgn}.
\end{proof}

\section{Asymptotic behavior of solution by means of the Lambert $W$-function}\label{section:asymptotic}

There is no non-trivial solution of \eqref{equation:a1b0} which would be expressed in terms of a power series of $1/x$ or in any similar form as that known for formal solutions of a linear differential or $q$-difference equation whose coefficients are analytic at infinity. In other words, one can check easily that \eqref{equation:a1b0} has no formal solution of the following form:
$$e^{P(x^{1/\nu})+Q(\log^2x)}x^\lambda\,\sum_{n\ge 0}a_n\,x^{-n/\nu},
 $$
 where $P(T)$ and $Q(T)$ are polynomial of $T$, $\nu$ denotes some positive integer, $\lambda\in\CC$ and where $a_n$ are complex numbers that are not all equal to zero.

In the following, we will consider the expansion at infinity of the function $u(x)$ defined as follows:
\begin{equation}\label{equation:asu}
u(x)=u(\lambda,x)=\int_\cC e^{v(x,t)}\,dt\,, \quad
\Re x>0\,,
\end{equation}
where $\cC$ denotes a Hankel type contour, and where
\begin{equation}\label{equation:asv}
v(x,t)=v(\lambda,x,t)=xt+\frac\lambda2\,\log^2t,\quad
\lambda>0.
\end{equation}
It will be shown that its asymptotic expansion can be expressed in terms of the Lambert $W$-function, which implies tha asymptotic bahavior of the solution $h$ of \eqref{equation:a1b0}; see \S \ref{subsection:asymp h}.

\subsection{Lambert $W$-function}\label{subsection:asW}
 Recall that $y=W(z)$ is defined to be the unique  solution analytic at $z=0$ of the following functional equation:
 \begin{equation}\label{equation:asW}
y\,e^{y}=z\,.
 \end{equation}
 This function is often used in the theory of delay-differential equations \cite{Wr1,Wr2,Wr3} and also is closely related to the tree generating function $T(z)$ popularized in the analysis of algorithms discipline \cite{CGHJK}.
The Taylor expansion at $0$ is given as follows:
$$
W(z)=\sum_{n\ge 1}\frac{(-n)^{n-1}}{n!}\,z^n\,,
$$
and $z=-1/e$ is the only branch point, so that the analytic domain of $W$ presents the universal covering of $\CC\setminus\{-1/e\}$.

\begin{proposition}[\cite{Wa}]\label{proposition:asWell}
 Let $\ell(z)$ denote the germ of analytic function at $z=0$ such that $\ell(0)=0$, $\ell'(0)=1$, and $\ell(z)=W(-e^{-1-\frac{z^2}2})+1$ in the neighborhood of $z=0$. Then $\ell$ represents an analytic function inside the disc $\vert z\vert<2\,\sqrt\pi\approx 3,5449$.
 \end{proposition}

\begin{proof}
Apply the inversion theorem to the relation $(1-y)e^y=e^{-\frac{z^2}{2}}$ near $y=0$ and $z=0$, noticing that the singularities of $y=\ell(z)$ are all non-zero complex numbers $z$ such that $e^{-\frac{z^2}2}=1$. Therefore, the nearest singularities of $\ell(z)$ may be $z_j=2\sqrt\pi\,e^{\frac{(2j+1)i}{4}}$, where $j=0$, $1$, $2$ or $3$. For more details, see \cite{Wa}.
\end{proof}

One can easily check that $y=\ell(z)$ satisfies the following nonlinear differential equation:
$$
yy'=(1-y)z\,,\quad
y(0)=0,\ y'(0)=1.
$$
If one writes $\ell(z)=\sum_{n\ge 0}c_nz^n$,  then $c_0=0$, $c_1=1$ and for all integer $n\ge 2$:
$$
c_{n}=-\frac1{n+1}\bigl({c_{n-1}}+\sum_{k=2}^{n-2}kc_kc_{n-k}\bigr)\,.
$$
 Thus, one gets easily the following values:
\begin{equation}\label{equation:asWc}
c_2=-\frac13,\quad
c_3=\frac1{36},\quad
c_4=\frac1{270},\quad
c_5=\frac1{4320},\quad \cdots.
\end{equation}

By the standard analytic continuation processus, one can notice the following
\begin{remark}\label{remark:asWell}\rm
The germ of analytic function $\ell$ considered in Proposition \ref{proposition:asWell} can be extended into the universal convering $\widetilde{\CC\setminus S}$, where
$$S=\cup_{k\in\ZZ^*}\{z\in\CC: z^2=4k\pi i\}.
$$
\end{remark}

For simplicity of exposition, while considering the asymptotics involving $W(z)$, we will restrict $W$ to the domain $|z|>1$ of the Riemann surface $\tilde\CC^*$ and write $L_2(z)=\log(\log z)$; moreover, if $\rho>1$,  we will denote $\Omega_\rho=\{z\in\tilde\CC^*: \vert z\vert>\rho\}$.

\begin{proposition}\label{proposition:asW}There exists $\rho_0>1$ such that the following relation holds in $\Omega_\rho$ for all $\rho>\rho_0$:
\begin{equation}\label{equation:asWas}
\bigl|W(z)-\bigl[\log z-L_2(z)+\frac{L_2(z)}{\log z}\bigr]\bigr|\le C_\rho\,\bigl|(\frac{L_2(z)}{\log z})^2\bigr|\,,
\end{equation}
where $C_\rho$ denotes a suitable positive constant depending upon $\rho$.
\end{proposition}

\begin{proof}
See {\cite[\S 2.4]{Br} and \cite[(4.19)]{CGHJK}}.
\end{proof}

In order to study the asymptotic behavior about the function $u(x)$ given by the integral \eqref{equation:asu},  it will be convenient  to introduce the following variant form $\omega$ of $W$:
\begin{equation}\label{equation:asomega}
\omega(x)=\omega(\lambda,x)=\lambda\,W(\frac x\lambda),
\end{equation}
where $\lambda>0$ as in \eqref{equation:asv} and where $x\in\Omega_{\rho}$ with $\rho>\lambda\rho_0>\lambda$.

From \eqref{equation:asW} one deduces that
\begin{equation}\label{equation:asomega1}
\omega(x)\,e^{\frac{\omega(x)}\lambda}=x\,;
\end{equation}
on the other hand, by taking \eqref{equation:asWas} into account, one finds the following expression for all $x\in \Omega_{\rho}$ with $\rho>\max(\lambda\rho_0,\rho_0)$:
\begin{equation*}\label{equation:asomega2}
\omega(x)=\lambda W(x)-\lambda\,\ln \lambda+\frac{\lambda\,\ln \lambda}{\log x}+O\bigl((\frac{L_2(x)}{\log x})^2\bigr)\,.
\end{equation*}
Finally, one can notice that, as $r\to+\infty$,
\begin{equation*}\label{equation:asomegamu}
\omega(re^{\mu i})=\omega(r)+\lambda\,\mu \bigl(1-\frac1{\ln r}\bigr)i+O\bigl((\frac{L_2(r)}{\ln r})^2\bigr)
\end{equation*}
uniformly for all $\mu\in\RR$. In particular, it follows that, for all $x\in\Omega_\rho$,
\begin{equation}\label{equation:asomegax}
\Re(\omega(x))=\omega(|x|)+o(1),\quad\vert x\vert\to+\infty.
\end{equation}

\subsection{Key lemma}\label{subsection:aslemma}  The function $u(x)=u(\lambda,x)$ given in \eqref{equation:asu} can be analytically continued into the whole complex plane excepted the negative real axis; consequently, $u$ will be considered for all $x\in\CC\setminus(-\infty,0]$.
We shall apply the saddle point method \cite[Chapter 5]{Br} to obtain the asymptotic expansion of that function.

Firstly, assume $x$ to be a large positive real number and check the saddle point of $v(x,t)$ along the integration path of $t$. Since $$
\partial_tv(x,t)=x+\lambda\frac{\log t}t\,,
$$
one obtains a unique positive real number $t=t_x$ such that $\partial_tv(x,t_x)=0$. By taking  \eqref{equation:asomega1} into account, one may find that $xt_x=\omega(x)=\omega(\lambda,x)$.

From now on,  one supposes that $x$ is a complex number of modulus $\rho>1$. If one writes
$$t=\frac{\omega(x)}x\,(1+s),
$$ a straightforward computation shows that the function $v(x,t)$ will take the following form:
\begin{equation}\label{equation:aslemmav}
v(x,t)=\frac{\omega(x)^2}{2\lambda}+\omega(x)\bigl(1-v_1(s)\bigr)+\frac{\lambda}2\,\log^2(1+s)\,,
\end{equation}
where
$$v_1(s)=-s+\log(1+s).
 $$
 The new integration path will be chosen to pass by the origin $s=0$ in a such manner that, near this point, the imaginary part of $v_1$ remains constant. If one writes $v_1(s)=\frac{\sigma^2}{2}\in[0,\infty)$, by Proposition \ref{proposition:asWell} the expected integration path can be parameterized  as follows:
\begin{equation}\label{equation:aslemmas}
 s=-\ell(i\sigma),\quad
- 2\sqrt\pi <\sigma<2\sqrt\pi\,.
\end{equation}

Let $\epsilon\in(0,2\sqrt\pi)$, and let $L=L_\epsilon$ denote the curve parameterized  by \eqref{equation:aslemmas} while the new parameter $\sigma$ describes the closed interval $[-2\sqrt\pi+\epsilon,2\sqrt\pi-\epsilon]$. If one writes
\begin{equation}\label{equation:asWPQ}
P=-\ell(i(2\sqrt\pi-\epsilon)),\quad
 Q=-\ell(i(-2\sqrt\pi+\epsilon)),
 \end{equation}
 then the curve $L$ will be assumed to start from $P$ and arrive at $Q$ by passing throughout the origin $O$ in the $s$-plane. For any $\delta\in(-\frac\pi2,\frac\pi2)$, let $L_{P,\delta}$ be the half straight-line ending at $P$ and having a angle equal to $\delta$ with the negative real-axis; similarly, let $L_{Q,\delta}$ denote the half straight-line starting from $Q$  whose angle with the negative real-axis  is equal to $\delta$. By this way, we will choose the integration path $\cC$ of \eqref{equation:asu} in a such manner that the transform
$$t\mapsto s=\frac x{\omega(x)}\,t-1$$ leads to a contour $L_{\delta_P,\delta_Q}$ given as follows:
$$
L_{\delta_P,\delta_Q}=L+L_{P,\delta_P}+L_{Q,\delta_Q},\quad
\delta_P,\ \delta_Q\in(-\frac\pi2,\frac\pi2).
$$

We are ready to write $u(x)$ into the sum of three integrals corresponding respectively to $L$, $L_{P,\delta_P}$ and $L_{Q,\delta_Q}$ in the $s$-plane. Indeed, Let
$$
A(x)=\frac{\omega(x)}x\,e^{\frac{\omega(x)^2}{2\lambda}+\omega(x)}\,;
$$
from \eqref{equation:asomega1} one deduces that
\begin{equation}\label{equation:aslemmaA}
A(x)=e^{\frac1{2\lambda}\,{\omega(x)^2}+(1-\frac1\lambda)\,\omega(x)}\,.
\end{equation}
Therefore, it follows that
\begin{equation}\label{equation:aslemmaU}
u(x)=A(x)\,\bigl[U(\omega(x))+U_P(\omega(x))+U_Q(\omega(x))\bigr]
\end{equation}
where $U=U_L$ is given by the following integral:
\begin{equation}\label{equation:aslemmaUL}
U_L(z)=\int_L e^{-zv_1(s)+\frac{\lambda}2\,\log^2(1+s)}ds,
\end{equation}
and where  $U_P$ and $U_Q$ are obtained by replacing the integration path $L$ by $L_{P,\delta_P}$ and $L_{Q,\delta_Q}$ respectively.

\begin{lemma}\label{lemma:aslemma}
Let $U$, $U_P$ and $U_Q$ be as appeared in \eqref{equation:aslemmaU}, and let $\mu_0\in(0,2\pi)$. Then the following properties hold.
\begin{enumerate}
\item The function $U(z)$ is an entire function that is bounded in the closed half-plane $\Re(z)\ge 0$ and is $O(e^{\mu_0|\Re(z)|})$ for $\Re(z)<0$. \item Each of the functions $U_P(z)$ and $U_Q(z)$ can be extended into an analytic function over the cut-plane $\CC\setminus(-\infty,0]$.
     \item For any $z_0>0$ and any $z_1>z_0$, the functions  $U_P(z)$ and $U_Q(z)$ are uniformly bounded by $C\,e^{-\mu_0z_0}$ in the half-plane $\Re(z)>z_1$, where $C$ denotes some positive constant depending only of the couple $(z_0,z_1)$.
    \item As $z\to\infty$ with $\Re(z)>0$, $U(z)$ admits an asymptotic expansion as follows:
    \begin{equation}\label{equation:aslemmaUas}
    U(z)=i\sqrt{\frac{2\pi}z} \,\bigl(1-(\frac1{12}+\frac\lambda2)\frac1z+(\frac1{288}-\frac\lambda{12}+\frac{3\lambda^2}{8})\frac1{z^2}+O(\frac1{z^3})\bigr).
    \end{equation}
\end{enumerate}
\end{lemma}

\subsection{Proof of Lemma \ref{lemma:aslemma}}\label{subsection:asproof} For any $\mu_0\in(0,2\pi)$, fix a positive real number $\epsilon$ in $(0,2\sqrt\pi)$ such that $\mu_0=\frac12\,{(2\sqrt\pi-\epsilon)^2}$. Thanks to the relation $v_1(s(\sigma))=\frac{\sigma^2}2$, the corresponding points $P$, $Q$ as given in \eqref{equation:asWPQ} satisfy the condition $$\mu_0=v_1(P)=v_1(Q).$$ Moreover, one can check that
$$\mu_0=\sup_{s\in L}v_1(s)\ge \sup_{s\in L_{P,\delta_P}\cup L_{Q,\delta_Q}}\Re(v_1(s)),
$$
where $\delta_P$, $\delta_Q\in(-\frac\pi2,\frac\pi2)$. On the other hand, $ v_1(s)\ge0$ for all $s\in L=L_\epsilon$; so one can easily find the assertions (1-3), and it remains to prove only the forth assertion and \eqref{equation:aslemmaUas}.

Let $s=-\ell(i\sigma)\in L$ as in \eqref{equation:aslemmas}. Since
$$
\log(1+s)=s+v_1(s)=-\ell(i\sigma)+\frac{\sigma^2}2\,,
$$
 the integral \eqref{equation:aslemmaUL} can be expressed as follows:
\begin{equation}\label{equation:asproofU}
U(z)=i\int_{-2\sqrt\pi\,+\epsilon}^{2\sqrt\pi\,-\epsilon}e^{-z\sigma^2/2}\,V(i\sigma)\,d\sigma\,,
\end{equation}
where $V(\sigma)$ denotes the analytic function defined in the disc $|\sigma|<2\sqrt\pi$  by the following relation:
$$
 V(\sigma)=V(\lambda,\sigma)=e^{\frac\lambda2(\ell(\sigma)^2+\sigma^2\ell(\sigma)+\frac{\sigma^4}4)}\,\ell'(\sigma)\,.
$$
The Taylor series coefficients \eqref{equation:asWc} of $\ell$ lead to that of $V$ at $\sigma=0$ as follows:
\begin{equation}\label{equation:asproofTaylor}
V(\sigma)=1-\frac23\,\sigma+V_2\,\sigma^2+V_3\,\sigma^3+V_4\,\sigma^4+...\,,
\end{equation}
where 
$$
V_2=\frac1{12}+\frac\lambda2\,,\quad
V_3=\frac2{135}-\frac\lambda6\,,\quad
V_4=\frac1{864}-\frac\lambda{36}+\frac{\lambda^2}8\,,\quad...
$$
On can easily see that the $n$th coefficient $V_n$ is a polynomial of degree $\frac n2$ in $\lambda$.

By taking the symmetry into account, one can write \eqref{equation:asproofU} into the following form:
\begin{equation}\label{equation:asproofU2}
U(z)=2i\int_0^{2\pi\,-\epsilon'}e^{-z\tau}\,\tilde V(i\sqrt{2\tau}\,)\,\frac{d\tau}{\sqrt{2\tau}}\,,
\end{equation}
where $\epsilon'=\epsilon(2\sqrt\pi\,-\frac\epsilon2)\in(0,2\pi)$, and $\tilde V(\sigma)=\frac12\bigl(V(\sigma)+V(-\sigma)\bigr)$. From \eqref{equation:asproofTaylor}, it follows that
\begin{equation}\label{equation:asproofV}
\tilde V (i\sqrt{2\tau})=1-(\frac1{12}+\frac\lambda2)(2\tau)+(\frac1{864}-\frac\lambda{36}+\frac{\lambda^2}8)(2\tau)^2+O(\tau^3)\,.
\end{equation}
Since
\begin{equation}\label{equation:asproofGamma}
\int_0^\infty e^{-\tau}\,\tau^{n-\frac12}\,{d\tau}=\Gamma(n+\frac12)=(n-\frac12)\,\cdot ... \cdot\,\frac12\,\cdot\sqrt\pi
\end{equation}
for all integer $n\ge 0$,  we ends the proof of \eqref{equation:aslemmaUas} by applying the Watson's lemma \cite[p. 614]{AAR} to the incomplete Laplace  integral \eqref{equation:asproofU2}. Therefore, we achieve the proof of Lemma \ref{lemma:aslemma}.
\hfill$\Box$

\subsection{Asymptotic behavior of $u(x)$ as $x\to\infty$}\label{subsection:asas}
Firstly, let us return to the function $\tilde V$ used in the integral \eqref{equation:asproofU2}. For any positive integer $N$,  instead of \eqref{equation:asproofV} we shall write
\begin{equation}\label{equation:ascommentV}
{\tilde V}(i\sqrt{2\tau}\,)=\Upsilon_N(\tau)+R_N(\tau),\quad \Upsilon_N(\tau)=\sum_{n=0}^NV_{2n}(-2\tau)^n,
\end{equation}
where  $R_N(\tau)=O(\tau^{N+1})$ for $\tau\to 0$ in $[0,2\pi)$, and where $V_{2n}$ is the Taylor coefficient of order $2n$ of the function $V$ as given in \eqref{equation:asproofTaylor}. Moreover, for any $\epsilon'\in(0,2\pi)$, we shall denote
$$
\Upsilon_{N,\epsilon'}(\tau)=\Upsilon_N(\tau)\,\chi_{[2\pi\,-\epsilon',+\infty)}+R_N(\tau)\,\chi_{[0,2\pi\,-\epsilon']}\,,
$$
where $\chi$ denotes the usual characteristic function.

The analyticity of $V$ inside the disc $\vert \tau\vert<2\sqrt\pi\,$ leads to geometric type estimates for the coefficients $V_{2n}$'s. Therefore, one can find positive constants $C$, $A>0$ such that, on the positive real-axis:
$$\vert\Upsilon_{N,\epsilon'}(\tau)\vert\le CA^N\tau^{N+1}.
$$
Let
\begin{equation}\label{equation:UN}
U_N(z)=1+\sum_{n=1}^N\frac{(2n)!\,V_{2n}}{2^n\,n!}\,\bigl(-\frac1z\bigr)^n\,.
\end{equation}
 Using \eqref{equation:asproofU2} and \eqref{equation:asproofGamma}, one can observe that
$$
U(z)-i\sqrt{\frac{2\pi}z}\,U_N(z)=\int_0^\infty e^{-z\tau}\,\Upsilon_{N,\epsilon'}(\tau)\,\frac{d\tau}{\sqrt{\tau}}.
$$
Thus the asymptotic expansion given in \eqref{equation:aslemmaUas} can be read as follows: there exist $C>0$ and $A\gtrsim 2\pi$ such that the following relation holds for any positive integer $N$ and any complex number $z$ with $\Re(z)>0$:
 \begin{equation}\label{equation:asasU}
\bigl\vert U(z)-i\sqrt{\frac{2\pi}z}\,U_N(z)\bigr\vert\le CA^N\frac{\Gamma(N+\frac32)}{(\Re(z))^{N+\frac32}}\,.
 \end{equation}

Finally, we are ready to prove the following

\begin{theorem}\label{theorem:asasu}
The function $u(x)=u(\lambda,x)$ defined in the half-plane $\Re(x)>0$ by \eqref{equation:asu}  can be extended into an analytic function on the Riemann surface $\tilde\CC^*$.

Moreover,  let $\omega(x)=\omega(\lambda, x)$ be as given in \eqref{equation:asomega} for $x\in\Omega_\rho\subset\tilde\CC^*$ with some $\rho>\lambda\rho_0$, and let
$$
u(x)=i\,e^{\frac1{2\lambda}\omega(x)^2+(1-\frac1\lambda)\omega(x)}\,\sqrt{\frac{2\pi}{\omega(x)}}\,\bigl(U_N(\omega(x))+r_N(x)\bigr)\,,
$$
where $U_N$ is given as in \eqref{equation:UN}.
Then there exists $C>0$ and $A\gtrsim 2\pi$ such that
the following estimates hold for all positive integer $N$ and all $x\in\Omega_\rho$:
\begin{equation}\label{equation:asasu}
|r_N(x)|\le CA^N\frac{N!}{(\omega(|x|))^{N+1}}\,.
\end{equation}

\end{theorem}

\begin{proof}If $x\in\RR^+$, one can write \eqref{equation:asu} as follows:
$$
u(x)=\frac1x\,\int_\cC e^{t+\frac\lambda2\log^2\frac tx}\,dt\,,
$$
what allows to make the analytic continuation over $\tilde\CC^*$.

In view of the relations \eqref{equation:asomegax}, \eqref{equation:aslemmaU} and \eqref{equation:asasU}, and by taking  Lemma \ref{lemma:aslemma} into account, it suffices to notice the following relation for any positive real $X>0$:
$$
\min_{N\ge 0}A^N\frac{N!}{X^{N+1}}\gtrsim C'e^{-A'X},
$$
where $C'>0$ and $A\gtrsim A'>0$, both being chosen independently of $X\in]0,X_0]$ for any given $X_0>0$.
\end{proof}

\subsection {Asymptotic expansion of the solution $h(x)$}\label{subsection:asymp h}
Consider the function $h$ given by \eqref{equation:b0h} on the Riemann surface $\tilde\CC^*$. By a suitable change of variables,  we can find immediately that
$$
h(x)=\frac{q\,\sqrt\kappa}{2\pi i}\,\int_\cC e^{xt-\frac1{2\ln q}\log^2(q^{\frac32}t)}\,dt\,,
$$
what gives rise to the following identity: for all $x\in\tilde\CC^*$,
\begin{equation}\label{equation:hu}
h(x)=\frac{u(-\frac1{\ln q},q^{-3/2}x)}{i\,\sqrt{2\pi q\ln (q^{-1})}}\,.
\end{equation}

\begin{theorem}\label{theorem:ashW}Let $V_{2}$, $V_4$, ... be the sequence defined by \eqref{equation:asproofTaylor} with $\lambda=-\frac1{\ln q}$, and let $h$ be the solution of \eqref{equation:a1b0} given as in \eqref{equation:b0h}. Then, we have the following asymptotic expansion as $x\to\infty$ in $\tilde\CC^*$:
\begin{equation}
h(x)=\frac1{\sqrt{q\ln(q^{-1})}}\,e^{-\frac{\ln q}2\tilde\omega^2+(1+\ln q)\tilde\omega}\,\frac1{\sqrt{\tilde\omega}}\,\bigl(1+\sum_{n=1}^N\frac{(2n)!\,V_{2n}}{2^n\,n!}\,
\bigl(-\frac1{\tilde\omega}\bigr)^n+O(\frac1{\tilde\omega^{N+1}})\bigr)\,,
\end{equation}
where $N$ denotes any positive integer and
$$
\tilde\omega=\tilde\omega(q,x)=\omega(-\frac1{\ln q},q^{-3/2}x)\,.
$$
\end{theorem}

\begin{proof}
It follows from Theorem \ref{theorem:asasu} combined with the identity \eqref{equation:hu}.
\end{proof}


\end{document}